%% file: main.tex
\documentclass[10pt]{tsoart}

\begin{document}

\title[Higher smooth structure sets of $\CC P^n$]{Higher smooth surgery structure sets \\ of complex projective spaces, part II}

\author{Samuel Kalu\v zn\'y}
\author{Tibor Macko}

\subjclass[2020]{Primary: 57R65, 57S25} 

\keywords{complex projective space, higher smooth surgery structure set, surgery, Arf invariant, stable stems, Toda brackets} 

\address{Faculty of Mathematics, Physics, and Informatics, Comenius University, Mlynsk\'a dolina,
SK-824 48, Bratislava, Slovakia} \email{tibor.macko@fmph.uniba.sk} \email{samuel.kaluzny@fmph.uniba.sk} 

\address{Institute of Mathematics, Slovak Academy of Sciences, \v Stef\'anikova 49, Bratislava, SK-81473, Slovakia} \email{macko@mat.savba.sk}

\thanks{This work was supported by grants VEGA 1/0425/25, UK/1192/2025, UK/3116/2024 and UK/197/2023. Parts of this work will be included in the PhD thesis of S.K} 

\date{\today}

\begin{abstract}
    We determine the torsion part of the higher smooth surgery structure sets of complex projective spaces (up to some extension problems) and complete the description of the forgetful map to their topological versions in low dimensions. In addition, an application to mapping class groups of complex projective spaces in low dimensions is presented.
\end{abstract}

\maketitle

\input{intro}

\input{theorems}

\input{norminv}

\newpage

\input{obstr}

\newpage

\input{proofs}

\bibliography{bibliography}  
\bibliographystyle{alpha}

\makeatletter
\enddoc@text     
\gdef\enddoc@text{}  
\makeatother

\newpage
\appendix
\input{appendix}

\end{document}

%% file: intro.tex
\section{Introduction}
\label{sec:introduction}

In part I \cite{kaluzny2026highersmoothsurgerystructure}, we proved that when $2n+k\geq5$, there are isomorphisms 
\begin{align}
\label{eqn:smooth-str-set-iso}
    \sS_{\partial}^{\DIFF}(\cpd)\cong\Z^{t_{n,k}}\oplus T_{n,k},
\end{align}
where 
\begin{align*}
    t_{n,k}=
        \begin{cases}
            0 & \text{if $k$ is odd},\\
            \lfloor\frac{n}{2}\rfloor+1 & \text{if $k\equiv0\mod4$ and $n$ is odd}, \\
            \lfloor\frac{n}{2}\rfloor & \text{otherwise},
        \end{cases}
\end{align*}
and $T_{n,k}$ is finite abelian; see \cite[Def. 2.1, Thm. 2.2]{kaluzny2026highersmoothsurgerystructure}. The goal of this article is to compute the torsion subgroups $T_{n,k}$ for $1\leq n,k\leq6$ and $2n+k\geq5$. 

Recall that, by a standard result, there are isomorphisms 
\begin{align} \label{eqn:top-str-set-of-cp-n-times-d-k}
    \textstyle\sS_{\partial}^{\TOP}(\cpdt)\xrightarrow{\cong}
        \Z^{t_{n,2l}}\oplus\Z_2^{n-t_{n,2l}} \quad \textup{and} \quad \sS_{\partial}^{\TOP} (\cp^n\times \text{D}^{2l+1})\xrightarrow{\cong} 0,
\end{align}
and analogous isomorphisms for $\TOP$-normal invariants, given by so-called splitting invariants; see \cite[Sec. 6]{kaluzny2026highersmoothsurgerystructure}. These fit, together with the isomorphisms obtained in the smooth case, into a diagram
\begin{equation}
\label{eqn:diagram-forgetful-map}
    {\footnotesize
    \begin{tikzcd}
        \Z^{t_{n,2l}}\oplus T_{n,2l} \arrow[r,"\cong"] \arrow[d,hookrightarrow,"E_{n,l}"] & \sS_{\partial}^{\DIFF}(\cpdt) \arrow{d}{\eta_{n,2l}^{\DIFF}} \arrow{r}{F_{\cpdt}} & \sS_{\partial}^{\TOP}(\cpdt) \arrow{d}{\eta_{n,2l}^{\TOP}} \arrow{r}{\cong} & \Z^{t_{n,2l}}\oplus\Z_2^{n-t_{n,2l}} \arrow[d,hookrightarrow] \\
        \Z^{t_{n,2l}'+\epsilon_{2l}}\oplus T_{n,2l}' \arrow[bend right=10,"F_{n,l}'"]{rrr} \arrow[r,"\cong"]  & \sN_{\partial}^{\DIFF}(\cpdt) \arrow{r}{} & \sN_{\partial}^{\TOP}(\cpdt) \arrow{r}{\cong} & \Z^{t_{n,2l}'+\epsilon_{2l}}\oplus\Z_2^{n+1-t_{n,2l}'-\epsilon_{2l}}
    \end{tikzcd}
    }
\end{equation}
Note that the isomorphisms~\eqref{eqn:smooth-str-set-iso} are not uniquely determined, choices are involved. After making the required choices, the forgetful map $F_{\cpdt}$ is described by a $(2\times2)$-block matrix $F_{n,l}$ given by a product:
\begin{equation}
\label{eqn:matrix-forgetful-map}
    \begin{pmatrix}
        A_{n,l} & 0 \\
        B_{n,l} & C_{n,l}
    \end{pmatrix}
    =F_{n,l}=F_{n,l}'\cdot E_{n,l}=
    \begin{pmatrix}
        A_{n,l}' & 0\\
        B_{n,l}' & C_{n,l}'
    \end{pmatrix}\cdot
    \begin{pmatrix}
        P_{n,l} & 0 \\
        Q_{n,l} & R_{n,l}
    \end{pmatrix}.
\end{equation}
We fix a specific isomorphism \eqref{eqn:smooth-str-set-iso} by making choices as specified in \cite[Remark 4.13]{kaluzny2026highersmoothsurgerystructure}, the table in Lemma~\ref{lema:coker-J-CP-values} and the proofs of Theorems~\ref{thm:main-theorem-3},\ref{thm:main-theorem-5} and \ref{thm:main-theorem-4} at the end of Chapter~\ref{sec:surgery-obstruction}. With respect to these choices, $A_{n,l}:\Z^{t_{n,2l}}\to\Z^{t_{n,2l}}$ was determined in \cite[Thm. 2.3, 2.4]{kaluzny2026highersmoothsurgerystructure} as a product $A_{n,l}'\cdot P_{n,l}$, up to an indeterminacy of index $2$ in $P_{n,l}$ when $n+l$ is odd. In this paper, we solve this indeterminacy and determine the matrices $B_{n,l}:\Z^{t_{n,2l}}\to\Z_2^{n-t_{n,2l}}$ and $C_{n,l}:T_{n,2l}\to\Z_2^{n-t_{n,2l}}$, thus giving a full characterization of $F_{\cpdt}$.

%% file: theorems.tex
\section{The Main Theorems} \label{sec:main-theorems}

In this section, we state the main theorems of this paper. In these statements, we suppose that $2n+k\geq5$.
\begin{thm}
\label{thm:main-theorem-1}
    The values of the torsion subgroups $T_{n,k}$ in the range $1\leq n,k\leq6$ are summarized in the following table up to some extension problems:
    \begin{table}[ht!]
        \centering
        \begin{tabular}{|c|c|c|c|c|c|c|} \hline
            \diaghead{\theadfont aaaaaaaaaa}%
            {$k$}{$n$} & $1$ & $2$ & $3$ & $4$ & $5$ & $6$ \\ 
            \hline
            $0$ & $\blank$ & $\blank$ & $\Z_2$ & $\Z_4$ & $\Z_2^2\oplus\Z_3$ & $\Z_2\oplus\Z_3$ \\
            \hline
            $1$ & $\blank$ & $\cdot$ & $\textcolor{blue}{\Z_4}$ & $\cdot$ & $\textcolor{blue}{\Z_2}$ & $\cdot$ \\
            \hline
            $2$ & $\blank$ & $\Z_2$ & $\Z_2^3$ & $\Z_2\oplus\Z_3$ & $\Z_2^2\oplus\Z_3$ & $\Z_2^3$ \\
            \hline
            $3$ & $\cdot$ & $\textcolor{blue}{\Z_2}$ & $\textcolor{blue}{\Z_2}\textcolor{red}{\oplus}\Z_2^2$ & $\textcolor{blue}{\Z_2}\textcolor{red}{\oplus}\Z_2$ & $\Z_2\oplus\Z_3$ & $\textcolor{blue}{\Z_2}\textcolor{red}{\oplus}\Z_2\oplus\Z_3$\\
            \hline
            $4$ & $\cdot$ & $\Z_4$ & $\Z_4\oplus\Z_2\oplus\Z_3$ & $\Z_4\oplus\Z_3$ & $\Z_4\oplus\Z_2\oplus\Z_3$ & $\Z_4^2\oplus\Z_3$ \\
            \hline
            $5$ & $\textcolor{blue}{\Z_{28}}$ & $\cdot$ & $\textcolor{blue}{\Z_{16}}$ & $\cdot$ & $\textcolor{blue}{\Z_8}\textcolor{red}{\oplus}\Z_2$ & $\Z_4$ \\
            \hline
            $6$ & $\Z_2^2$ & $\Z_3$ & $\Z_2\oplus\Z_3$ & $\Z_2^2$ & $\Z_2^4$ & $\Z_4\oplus\Z_2^2$ \\
            \hline
        \end{tabular}
        \caption{\centering Table of $T_{n,k}$ (extension problems are highlighted in red, elements from $L_{2n+k+1}(\Z)$ in blue)}
        \label{tabletorsCP}
\end{table}
\end{thm}
In the next theorems, we characterize the forgetful map $F_{\cpdt}$ in the range $1\leq n,2l\leq6$; see Diagram \eqref{eqn:diagram-forgetful-map} and \eqref{eqn:matrix-forgetful-map}. The characterization is in terms of matrices with respect to choices described above.
\begin{thm}
\label{thm:main-theorem-2}
    For $l=1$, the entries in $B_{n,l}'$ are summarized in Table~\ref{tab:matrix-B_n,1'}. For $l=2,3$, we have: 
    \begin{align}
    \label{eqn:values-of-Bn,l}
        B_{n,l}'= 0_{(t_{n,2l}'+\epsilon_{2l})\times(n+1-t_{n,2l}'-\epsilon_{2l})}.
    \end{align}
    The entries in the matrices $C_{n,l}'$ are summarized in Table \ref{tab:matrix-C_n,l'}.
\end{thm}
The matrix $P_{n,l}$ from \eqref{eqn:matrix-forgetful-map} was already determined in \cite[Thm. 2.4]{kaluzny2026highersmoothsurgerystructure}, up to an indeterminacy of index $2$. Here, we solve this indeterminacy. We also compute the matrices $Q_{n,l}$ and $R_{n,l}$.
\begin{thm}
\label{thm:main-theorem-3}
    For odd $n+l$, $P_{n,l}$ is the identity matrix.
\end{thm}
\begin{thm}
\label{thm:main-theorem-5}
    The matrix $Q_{n,l}$ is zero unless $l=1$ and $n$ is even, when its entries are as summarized in Table~\ref{tab:matrix-Q_n,1}.
\end{thm}
\begin{thm}
\label{thm:main-theorem-4}
    If $n+l$ is even, $R_{n,l}$ is the identity matrix. For odd $n+l$, its entries are summarized in Table \ref{tab:matrix-R_n,l}.
\end{thm}
The following corollary is an improvement of \cite[Corollary 2.7]{kaluzny2026highersmoothsurgerystructure}.
\begin{cor}
\label{cor:splitting-invariants}
    Let $(f,\partial f):(X,\partial X)\to(\cp^6\times D^{2l},\cp^6\times S^{2l-1})$ represent an element of $\sS^{TOP}(\cp^6\times D^{2l})$, $1\leq l\leq3$. The map $(f,\partial f)$ admits a smooth structure on $X$ if and only if its splitting invariants $(\overline{\sigma}_{0,2l},\dots,\overline{\sigma}_{5,2l})$ satisfy the congruences given in Table~\ref{tab:splitting-invariants}.
\end{cor}
\begin{remark}[Smooth Mapping Class Groups]
\

    Here, we present an application of our results communicated to us by Oscar Randal-Williams, for which we are thankful (see also \cite{randal-williams}). 

    Let $X$ be a closed smooth manifold with $\pi_1(X)=0$. We denote by $G(X)$ the space of self-homotopy equivalences of $X$, and by $\CAT(X)$ the space of $\CAT$-automorphisms of $X$; see \cite[Appendix 1.2.a]{burghelea2006}. In particular,
    \begin{equation}
    \label{eq:block-CAT-mapping-class-group}
        \pi_0\CAT(X)\cong\MCG^{\CAT}(X),
    \end{equation}
    the $\CAT$-mapping class group of $X$. Further, we denote by $\widetilde{G}(X)$ and $\widetilde{\CAT}(X)$ the spaces of block self-homotopy equivalences and block $\CAT$-automorphisms of $X$; see \cite[1.1.4.]{weiss-williams-automorph} for a definition. There are homotopy fiber sequences 
    \begin{equation}
    \label{eq:LES-block-automorph-spaces}
        \widetilde{\CAT}(X)\to\widetilde{G}(X)\to\widetilde{\sS}^{\CAT}(X),
    \end{equation}
    where $\widetilde{\sS}^{\CAT}(X)$ is the block $\CAT$-structure space of $X$; see again \cite[1.1.4.]{weiss-williams-automorph}. The homotopy groups of $\widetilde{\sS}^{\CAT}(X)$ satisfy
    \[
        \pi_k\widetilde{\sS}^{\CAT}(X)\cong\sS_{\partial}^{\CAT}(X\times\text{D}^k);
    \]
    see \cite[Ch.17.A]{Wall(1999)}. There is an obvious forgetful map $\widetilde{\DIFF}(X)\to\widetilde{\TOP}(X)$ which relates the fiber sequences~\eqref{eq:LES-block-automorph-spaces}, and applying $\pi_*(\blank)$ we obtain for $X=\cp^n$ the following diagram, where the columns and rows are exact:
    \begin{equation*}
        \begin{tikzcd}
             & \pi_1\widetilde{G}(\cp^n) \arrow{r}{} \arrow{d}{} & \pi_1\widetilde{G}(\cp^n) \arrow{d}{} \\
            {[Th(\cp^n\times\text{D}^1),\TOP/\Or]} \arrow{r}{} \arrow{d}{} & \sS_{\partial}^{\DIFF}(\cp^n\times\text{D}^1) \arrow{r}{F_{\cp^n\times\text{D}^1}} \arrow{d}{} & \sS_{\partial}^{\TOP}(\cp^n\times\text{D}^1) \arrow{d}{} \\
            \pi_1(\widetilde{\DIFF}(\cp^n)/ \widetilde{\TOP}(\cp^n)) \arrow{r}{} &  \pi_0\widetilde{\DIFF}(\cp^n) \arrow{r}{} \arrow{d}{} & \pi_0\widetilde{\TOP}(\cp^n) \arrow{d}{} \\ 
             & \pi_0\widetilde{G}(\cp^n) \arrow{r}{} & \pi_0\widetilde{G}(\cp^n)
        \end{tikzcd}
    \end{equation*}
    We know from~\eqref{eqn:top-str-set-of-cp-n-times-d-k} that $\sS_{\partial}^{\TOP}(\cp^n\times\text{D}^1)\cong0$. This implies that the middle map in the right column is injective. Next, from \cite[1.1.4.]{weiss-williams-automorph} we know that the inclusion $G(\cp^n)\to\widetilde{G}(\cp^n)$ is a homotopy equivalence. Here, $\pi_0G(\cp^n)\cong\Z_2$ is generated by complex conjugation $c:\cp^n\xrightarrow{\cong}\cp^n$; see \cite[Ex.~18.16]{Lueck-Macko(2024)}, and $\pi_1G(\cp^n)\cong\Z_{n+1}$ where the generator comes from an action of the projective unitary group $U(n+1)/\Delta_{n+1}$; see \cite[Thm.~1.1, Prop.~1.2]{sasao-homotopy-of-map-cp}. In particular, the map 
    \[
        \pi_1\widetilde{\DIFF}(\cp^n)\to\pi_1G(\cp^n)\cong\Z_{n+1}
    \]
    is surjective. Finally, since complex conjugation is smooth, the map 
    \[
        \pi_0\widetilde{\TOP}\to\pi_0\widetilde{G}(\cp^n)\cong\Z_2\{c\}
    \]
    is also surjective, so it is an isomorphism. We obtain the following diagram:
    \begin{equation*}
        \begin{tikzcd}
             & \Z_{n+1} \arrow[r,phantom,"\cong"] \arrow{d}{0} & \Z_{n+1} \arrow{d}{} \\
            {[Th(\cp^n\times\text{D}^1),\TOP/\Or]} \arrow{r}{} \arrow[d,phantom,sloped,"\cong"] & \sS_{\partial}^{\DIFF}(\cp^n\times\text{D}^1) \arrow{r}{} \arrow[>->]{d} & 0 \arrow{d}{} \\
            \pi_1(\widetilde{\DIFF}(\cp^n)/ \widetilde{\TOP}(\cp^n)) \arrow{r}{} &  \pi_0\widetilde{\DIFF}(\cp^n) \arrow{r}{} \arrow[->>]{d} & \Z_2\{c\} \arrow[d,phantom,sloped,"\cong"] \\ 
             & \Z_2\{c\} \arrow[r,phantom,"\cong"] & \Z_2\{c\}
        \end{tikzcd}
    \end{equation*}
    In the simplicial model of the spaces $\DIFF(X)$ and $\widetilde{\DIFF}(X)$ from \cite[Appendix 1.2.a,1.3.a]{burghelea2006}, the $1$-simplices are given by the relation of isotopy and pseudo-isotopy. We know from \cite[Intro, Corollary 1]{cerf1970} that if $\dim X\geq5$ then the resulting relations on the $0$-simplices coincide. Thus, 
    \[
        \pi_0\widetilde{\DIFF}(\cp^n)\cong\pi_0\DIFF(\cp^n)
    \]
    for $n\geq3$. From \eqref{eq:block-CAT-mapping-class-group} and the middle column of the diagram we obtain the following short exact sequence: 
    \begin{equation*}
        0\to\sS_{\partial}^{\DIFF}(\cp^n\times\text{D}^1)\to\MCG^{\DIFF}(\cp^n)\to\Z_2\{c\}\to0
    \end{equation*}
    This sequence splits since $c$ is a diffeomorphism. From \cite[IV.2.1]{brown-group-cohomology} we obtain 
    \begin{equation}
    \label{eq:mapping-class-group-cp}
        \MCG^{\DIFF}(\cp^n)\cong\sS_{\partial}^{\DIFF}(\cp^n\times\text{D}^1)\rtimes\Z_2\{c\}
    \end{equation}
    where the semi-direct product is determined by the conjugation action by $c$ on the image of $\sS_{\partial}^{\DIFF}(\cp^n\times\text{D}^1)$ in $\MCG^{\DIFF}(\cp^n)$. Direct verification shows that the conjugation action by $c$ and the action by post-composition with $c$ agree on $\sS_{\partial}^{\DIFF}(\cp^n\times\text{D}^1)$.

    Since by \cite[Sec.17.7, p.798-799]{Lueck-Macko(2024)} the action by post-composition with $c$ on elements coming from $L_{2n+2}(\Z)$ is $(-1)^n\id$, we obtain from Theorem~\ref{thm:main-theorem-1} and \eqref{eq:mapping-class-group-cp}:
    \begin{align*}
        &\MCG^{\DIFF}(\cp^3)\cong\text{D}_8,\quad\MCG^{\DIFF}(\cp^5)\cong\Z_2\{\Sigma^{11}\}\oplus\Z_2\{c\}, \\
        &\MCG^{\DIFF}(\cp^4)\cong\MCG^{\DIFF}(\cp^6)\cong\Z_2\{c\}.
    \end{align*}
    Here, $\text{D}_8$ is the dihedral group on $8$ elements and $\Sigma^{11}$ is the connected sum with the generator of $bP_{12}\cong\Theta_{11}$. The subgroup represented by orientation-preserving diffeomorphisms for $n=3$ was already obtained by Brumfiel in \cite[Remark II.11]{Brumfiel1971}, and reproduced by Kreck and Su in \cite{kreck2024}. It seems that the knowledge of the full $\MCG^{\DIFF}(\cp^3)$ is folklore; see the comments to \cite{randal-williams}. 

    We could extend the results from Lemma~\ref{lema:coker-J-CP-values} and combine them with \cite[Table 5]{kaluzny2026highersmoothsurgerystructure} to obtain
    \[
        \MCG^{\DIFF}(\cp^7)\cong(\Z_8\{\Sigma^{15}\}\textcolor{red}{\oplus}\Z_2\{\eta\kappa\})\rtimes\Z_2\{c\},
    \]
    where $\eta\kappa$ comes from $\coker[\Sigma\cp^7,\Omega J]\cong\sN_{\partial}^{\DIFF}(\cp^7\times\text{D}^1)$, $\Sigma^{15}$ is the connected sum with the generator of $bP_{16}\subseteq\Theta_{15}$. The red color highlights an unsolved extension problem.
\end{remark}

\begin{remark}[Methods and Organization]
\label{rem:methods-and-organization}
\ 

    As described in \cite[Sec.3]{kaluzny2026highersmoothsurgerystructure}, determining the torsion group $T_{n,k}$ proceeds in two steps. First, calculate the normal invariants, and second, evaluate the surgery obstruction. The main tool in the first part is the spectral sequence of Brumfiel from \cite[Sec.5]{Brumfiel-(1970))} which is used to calculate the torsion part of the normal invariants. The surgery obstruction is given by the Arf invariant of manifolds with boundary.

    The spectral sequence, which calculates $[\Sigma^k\cp^n,Y]$ for any space $Y$, is explained in detail in Section \ref{sect:normal-invariants}. Brumfiel calculated parts of the sequences for $Y=G$ and $G/O$ in a certain range of dimensions. It turns out that when extending the results of Brumfiel to $k > 0$, essentially the same methods work. However, the relationship between the sequences for $G$ and $G/O$ turns out to be more complex, which we address in Subsection~\ref{subsect:results-spectral-sequence}.

    We also provide detailed arguments where Brumfiel only states the results (especially at the prime $p=2$). Moreover, we solve the extension problems (some of which were also recently solved by Kasilingam in \cite{kasilingam2026diffeomorphismclassificationsmoothstructures}). In order to solve the extension problems we need to include in Section~\ref{sect:normal-invariants} a thorough explanation of the differentials and the necessary stable homotopy theory. We also use more recent results such as the computation of James numbers from \cite{lundell}, or computations of stable self-maps of complex projective spaces from \cite{Kachi2001SomeCG}, to clarify some parts of the calculations. In Subsection~\ref{subsect:results-spectral-sequence}, we use Sullivan's splittings of $G/O$ localized at a prime to clarify the above-mentioned relationship between the sequences for $G$ and $G/O$.

    As for the surgery obstructions presented in Section \ref{sec:surgery-obstruction}, we use the formula for the Arf invariant of manifolds with boundary; see \cite[Eq.(0.6)]{Hambleton}. The formula is in terms of characteristic classes, which were extensively studied by Brumfiel, Madsen and Milgram; see \cite{brumfiel_PL}, \cite{brumfiel_transfer} and \cite{Madsen}. Their results are used, together with some of our own calculations and results from Part I \cite{kaluzny2026highersmoothsurgerystructure}, to compute the desired surgery obstructions.

    The obstacles for extending our calculations to higher $k$ and $n$ are as follows. The spectral sequence for computing the normal invariants works in principle for any $k$ and $n$. However, to obtain specific values for higher $k$ and $n$ one needs to calculate more pages of the sequence and their larger portions. This requires computations of stable homotopy groups of spheres in a wider range (which is available roughly up to dimension $100$) and stable homotopy types of stunted complex projective spaces (which seems to be less available). The surgery obstruction formula also works in principle for any $k$ and $n$, and is significantly simplified for $k>0$. However, to effectively use it, one would need to know the Kervaire invariants of the elements resulting from the spectral sequence, as well as extend calculations of Adams operations and Stiefel-Whitney classes of bundles over $\Sigma^k\cp^n$.

    Proofs of the main theorems and Corollary~\ref{cor:splitting-invariants} are contained in Section~\ref{sec:proofs}. The key technical statements from previous sections are Lemmas~\ref{lema:coker-J-CP-values}, \ref{obstrfreeodd}, \ref{obstrtors} and \ref{lema:Arf-inv-on-sphere-part}.
\end{remark}

\begin{remark}
    The table in Theorem~\ref{thm:main-theorem-1} still contains unsolved extension problems. However, these are of a different nature than those mentioned in Remark~\ref{rem:methods-and-organization}. Namely, they arise from the long exact surgery sequence for $\cp^n$; see \cite[Eq.(3.2)]{kaluzny2026highersmoothsurgerystructure}. Their solution will probably require different technology from that used in the present paper.
\end{remark}

%% file: norminv.tex
\section{The Normal Invariants}
\label{sect:normal-invariants}

Recall that in \cite[Lemma 4.1, Corollary 4.2]{kaluzny2026highersmoothsurgerystructure} we have shown that the normal invariants of $\cpd$ split as 
\begin{equation}
\label{eq:splitting-norm-inv-thom-space}
    [Th(\cpd),G/O]\cong\ker[\Sigma^k\cp^n,J]\oplus \coker[\Sigma^k\cp^n,\Omega J]\oplus\pi_k(G/O),
\end{equation}
and that $\ker[\Sigma^k\cp^n,J]\cong\Z^{t_{n,k}'}$ (where $t_{n,k}'$ is closely related to $t_{n,k}$; see \cite[Remark 4.3]{kaluzny2026highersmoothsurgerystructure}). In addition, we have obtained specific generators of this free abelian subgroup as stable (virtual) bundles from $[\Sigma^k\cp^n,BO]$. In this section, we compute the finite abelian subgroup $\coker[\Sigma^k\cp^n,\Omega J]$ for $n,k\leq6$. Our primary tool will be a spectral sequence used by Brumfiel in \cite{Brumfiel-(1970))} to compute $[\cp^n,G]$ for $n\leq6$.

Consider the cofibration sequences $$\textstyle S^{2n+1}\xrightarrow{H_n}\cp^n\xrightarrow{i_n}\cp^{n+1}\xrightarrow{j_n}S^{2n+2}\xrightarrow{\Sigma H_n}\Sigma\cp^n\xrightarrow{\Sigma i_n}\dots,$$one for each $n\in\Z$. These cofibration sequences induce long exact sequences of groups $$\textstyle \pi_{2n+1}(Y)\xleftarrow{[H_n,Y]}[\cp^n,Y]\xleftarrow{[i_n,Y]}[\cp^{n+1},Y]\xleftarrow{[j_n,Y]}\pi_{2n+2}(Y)\xleftarrow{[\Sigma H_n,Y]}\dots$$for any space $Y$, which assemble into the following exact couple:
\begin{equation}
    \label{exactcouple}
    \begin{tikzcd}
    \displaystyle\prod_{n,k\geq0}[\textstyle\Sigma^k\cp^n,Y]
        \arrow{dr}[swap]{\prod_{n,k\geq0}[\Sigma^{k} H_{n},Y]} 
    & & 
    \displaystyle\prod_{n,k\geq0}[\textstyle\Sigma^k\cp^{n+1},Y] 
         \arrow{ll}[swap]{\prod_{n,k\geq0}[\Sigma^{k} i_{n},Y]}\\
    & 
    \displaystyle\prod_{n,k\geq0}\pi_{2n+k+1}(Y) 
         \arrow{ur}[swap]{\prod_{n,k\geq0}[\Sigma^{k} j_{n},Y]}
    &
\end{tikzcd}
\end{equation}
There is a spectral sequence associated to this exact couple in the standard way; see \cite[Subsec.2.2]{mccleary2001user} (especially the part titled \textit{Exact couples}). The differentials on the first page 
\begin{equation}
\label{eq:differentials-1st-page}
    d_1^{k,n}(Y):\pi_{2n+k+1}(Y)\cong E_1^{k,n}(Y)\to E_1^{k-1,n+1}(Y)\cong\pi_{2n+k+2}(Y)    
\end{equation}
are given by the composition 
\begin{equation}
\label{eq:1st-differential-composition}
    \pi_{2n+k+1}(Y)\xrightarrow{[\Sigma^{k-1}j_n,Y]}[\textstyle\Sigma^{k-1}\cp^{n+1},Y]\xrightarrow{[\Sigma^{k-1}H_{n+1},Y]}\pi_{2(n+1)+(k-1)+1}(Y).    
\end{equation}
The second page consists of homology groups $E_2^{k,n}(Y)=\ker d_1^{k,n}(Y)/\im d_1^{k+1,n-1}(Y)$. We obtain a new exact couple, the derived couple; see \cite[Prop.2.7]{mccleary2001user}. Iterating this process, we get pages $E_r^{k,n}(Y)$ for each $r\in\N$ that are quotients of subgroups of $\pi_{2n+k+1}(Y)$, and differentials on the $r$-th page $d_r^{k,n}(Y):E_r^{k,n}(Y)\to E_r^{k-1,n+r}(Y)$.

In this case, we are not interested in the $E_{\infty}$-page. Our goal is to calculate the groups $[\Sigma^{k}\cp^n,Y]$ that appear in the original exact couple \eqref{exactcouple}. These fit into the following short exact sequences: 
\begin{align}
\label{eq:exact-sequence-from-couple}
    \textstyle0\to\ker[\Sigma^{k}i_n,Y] \xrightarrow{[\Sigma^{k}j_n,Y]}[\Sigma^{k}\cp^{n+1},Y]\xrightarrow{[\Sigma^{k}i_n,Y]}[\Sigma^{k}\cp^n,Y] \\ \nonumber
    \xrightarrow{[\Sigma^{k}H_n,Y]}\im[\Sigma^{k}H_n,Y]\to0.
\end{align}

We seek to express the outer groups in the sequence in terms of differentials $d_r^{k,n}(Y)$. Define $D_i^{k,n}(Y)\subseteq\pi_{2n+k}(Y)$ for $i=1$ as 
\[
    D_1^{k,n}(Y):=\im d_1^{k,n-1}(Y)\subseteq E_1^{k-1,n}(Y)\cong\pi_{2n+k}(Y),
\]
and for $i>1$ inductively as the preimage of $\im\alpha_i^{k,n}$ under the third vertical map in the middle short exact sequence of the following diagram 
\[
\begin{tikzcd}
    0 \arrow{d}{} & 0 \arrow{d}{} & \\
    D_{i-1}^{k,n}(Y) \arrow[r,equal] \arrow{d}{} & D_{i-1}^{k,n}(Y) \arrow{d}{\delta_{i-1}^{k,n}} \arrow{r}{} & 0 \arrow{d}{} \\
    D_i^{k,n}(Y) \arrow{d}{} \arrow{r}{\delta_i^{k,n}} & \ker d_1^{k-1,n}(Y) \arrow{d}{} \arrow{r}{} & \ker d_1^{k-1,n}(Y)/\im\delta_i^{k,n} \arrow{d}{\cong} \\
    \im d_i^{k,n-i}(Y) \arrow{d}{} \arrow{r}{\alpha_{i}^{k,n}} & \ker d_1^{k-1,n}(Y)/\im\delta_{i-1}^{k,n} \arrow{d}{} \arrow{r}{} & \left(\ker d_1^{k-1,n}(Y)/\im\delta_{i-1}^{k,n}\right)/\im d_i^{k,n-i}(Y) \arrow{d}{} \\
    0 & 0 & 0 
\end{tikzcd}
\] 
where $\delta_{i-1}^{k,n}$ exists by induction and $\alpha_i^{k,n}$ is the composition 
\begin{align*}
    \im d_i^{k,n-i}(Y)&\hookrightarrow E_{i}^{k-1,n}(Y)\cong\ker d_{i-1}^{k-1,n}(Y)/\im d_{i-1}^{k,n-(i-1)}(Y) \\ 
    &\hookrightarrow E_{i-1}^{k-1,n}(Y)/\im d_{i-1}^{k,n-(i-1)}(Y) \\
    &\hookrightarrow (E_{i-2}^{k-1,n}(Y)/\im d_{i-2}^{k,n-(i-2)}(Y))/\im d_{i-1}^{k,n-(i-1)}(Y) \\
    &\vdots \\
    &\hookrightarrow(\ker d_1^{k-1,n}(Y)/\im d_1^{k,n-1}(Y))/\im d_{2}^{k,n-2}(Y)/\dots/\im d_{i-1}^{k,n-(i-1)}(Y) \\
    &\cong\ker d_1^{k-1,n}(Y)/\im\delta_{i-1}^{k,n}
\end{align*}
The last isomorphism follows (inductively) from the right column of the diagram. Since $\im d_i^{k,n-i}(Y)=0$ for $i\geq n$, the induction stabilizes and we denote 
\[
    D^{k,n}(Y):=D_n^{k,n}(Y).
\]
\begin{lema}
\label{sesspec}
    The short exact sequence \eqref{eq:exact-sequence-from-couple} is isomorphic to
    \begin{align*}
        \textstyle0\to\pi_{2n+k+2}(Y)/D^{k+2,n}(Y)\xrightarrow{[\Sigma^{k}j_n,Y]}[\Sigma^{k}\cp^{n+1},Y]\xrightarrow{[\Sigma^{k}i_n,Y]}[\Sigma^{k}\cp^n,Y] \\ 
        \xrightarrow{[\Sigma^{k}H_n,Y]}D^{k+1,n}(Y)\to0
    \end{align*}
\end{lema}
\begin{proof}
    We need to prove that 
    \[
        \ker[\Sigma^{k}i_n,Y]\cong\pi_{2n+k+2}(Y)/D^{k+2,n}(Y)\quad\text{and}\quad \im[\Sigma^{k}H_n,Y]\cong D^{k+1,n}(Y).
    \]
    From \eqref{exactcouple} it follows that 
    \begin{align*}
        \ker[\Sigma^{k}i_n,Y]&\cong\im[\Sigma^{k}j_n,Y]\cong \pi_{2n+k+2}(Y)/\ker[\Sigma^{k}j_n,Y] \\ 
        &\cong\pi_{2n+k+2}(Y)/\im[\Sigma^{k+1}H_n,Y].
    \end{align*}
    Thus, it is sufficient to prove that $\im[\Sigma^kH_n,Y]\cong D^{k+1,n}(Y)$ for $k\in\N_0$.

    Now, we take a closer look at the term $D^{k+1,n}(Y)$.The differentials on the first page are defined in \eqref{eq:1st-differential-composition} as $d_1^{k+1,n-1}(Y)=[\Sigma^{k}H_{n},Y]\circ[\Sigma^{k}j_{n-1},Y].$
    Since $\im[\Sigma^{k}j_{n-1},Y]\cong\ker[\Sigma^{k}i_{n-1},Y]$, we get 
    \[
        \im d_1^{k+1,n-1}(Y)\cong[\Sigma^{k}H_{n},Y](\ker[\Sigma^{k}i_{n-1},Y]).
    \]
    On the second page, the differentials are defined as $d_2^{k+1,n-2}(Y)=[\Sigma^kH_n,Y]\circ[\Sigma^kj_{n-2},Y]$ where (by abuse of notation)
    \begin{align*}
        [\Sigma^kj_{n-2},Y]:E_2^{k+1,n-2}(Y)&\to[\Sigma^ki_{n-1},Y](\textstyle[\Sigma^k\cp^n,Y])\subseteq[\Sigma^k\cp^{n-1},Y]\\
        [x]&\mapsto[\Sigma^kj_{n-2},Y](x)  
    \end{align*}
    and 
    \begin{align*}
        [\Sigma^kH_n,Y]:[\Sigma^ki_{n-1},Y](\textstyle[\Sigma^k\cp^n,Y])&\to E_2^{k,n}(Y) \\
        [\Sigma^ki_{n-1},Y](y)&\mapsto[[\Sigma^kH_n,Y](y)]
    \end{align*}
    So $d_2^{k+1,n-2}(Y)([x])$ is the class in $E_2^{k,n}(Y)=\ker d_1^{k,n}(Y)/\im d_1^{k+1,n-1}(Y)$ represented by 
    \begin{align}
    \label{eq:representative-of-2nd-differential}
        [\Sigma^kH_n,Y]\{[\Sigma^ki_{n-1},Y]^{-1}([\Sigma^kj_{n-2},Y](x))\}
    \end{align}
    Again, using the exactness of \eqref{exactcouple} this can be rewritten as 
    \[
        \im d_2^{k+1,n-2}(Y)=[\Sigma^kH_n,Y]\{[\Sigma^ki_{n-1},Y]^{-1}(\ker[\Sigma^ki_{n-2},Y])\}\subseteq E_2^{k,n}(Y)
    \]
    Now $[\Sigma^ki_{n-1},Y]^{-1}\ker[\Sigma^ki_{n-2},Y]=\ker([\Sigma^ki_{n-2},Y]\circ[\Sigma^ki_{n-1},Y])$ (from now on, we denote the second term $\ker[\Sigma^{*}i_{*},Y]^2$) and we obtain 
    \[
        \im d_2^{k+1,n-2}(Y)=[[\Sigma^kH_n,Y](\ker[\Sigma^{*}i_{*},Y]^2)]. 
    \]
    This class is determined up to $\im d_1^{k+1,n-1}(Y)=[\Sigma^{k}H_{n},Y](\ker[\Sigma^{k}i_{n-1},Y])$, therefore
    \[
        \im d_2^{k+1,n-2}(Y)\cong\frac{[\Sigma^kH_n,Y](\ker[\Sigma^{*}i_{*},Y]^2)}{[\Sigma^{k}H_{n},Y](\ker[\Sigma^{k}i_{n-1},Y])}.
    \]
    Similarly, for differentials on higher pages we get 
    \[
        \im d_r^{k+1,n-r}(Y)\cong\frac{[\Sigma^{k}H_{n},Y](\ker[\Sigma^{*}i_{*},Y]^r)}{[\Sigma^{k}H_{n},Y](\ker[\Sigma^{*}i_{*},Y]^{r-1})}. 
    \]
    Since $[\Sigma^k\cp^n,Y]$ (the source of $[\Sigma^kH_n,Y]$) belongs to $\ker[\Sigma^{*}i_{*},Y]^N$ for $N>n$, we see that $D^{k+1,n}(Y)$ gives a filtration of $\im[\Sigma^kH_n,Y]$.
\end{proof}
In order to use Lemma \ref{sesspec} to compute $[\Sigma^{2k}\cp^n,Y]$, we need to determine the differentials. By \eqref{eq:1st-differential-composition}, 
\begin{align*}
    d_1^{k,n}(Y):\pi_{2n+k+1}(Y)&\to\pi_{2n+k+2}(Y) \\
    [\alpha]&\mapsto[\alpha_1]
\end{align*}
where $\alpha_1$ is given by the diagram
\[
\begin{tikzcd}
    S^{2n+k+2} \arrow[dotted]{r}{\alpha_1} \arrow{d}{\Sigma^{k-1}H_{n+1}} & Y \\
    \Sigma^{k-1}\cp^{n+1} \arrow{r}{\Sigma^{k-1}j_{n}} & S^{2n+k+1} \arrow{u}{\alpha}
\end{tikzcd}
\]

Next, by \eqref{eq:representative-of-2nd-differential} 
\[
    \begin{tikzcd}
        \ker d_1^{k,n}(Y)/\im d_1^{k+1,n-1}(Y) & \ker d_1^{k-1,n+2}(Y)/\im d_1^{k,n+1}(Y) \\
        \mathllap{d_2^{k,n}(Y){:}\ } E_2^{k,n}(Y) \arrow{r}{} \arrow[u,phantom,sloped,"\cong"] & E_2^{k-1,n+2}(Y) \arrow[u,phantom,sloped,"\cong"] \\
        {[\alpha]} \arrow[r,maps to] & {[\alpha_2]}
    \end{tikzcd}  
\]
where $\alpha_2$ is given by the following diagram
\[
\begin{tikzcd}
    |[minimum width=2cm]| S^{2n+k+4} \arrow{r}{\Sigma^{k-1}H_{n+2}} \arrow[bend left, dotted]{rr}{\alpha_2} & |[minimum width=2.5cm]| \Sigma^{k-1}\cp^{n+2} \arrow{r}{\overline{\alpha}} & Y \\
    &\Sigma^{k-1}\cp^{n+1} \arrow{u}{\Sigma^{k-1}i_{n+1}} \arrow{r}{\Sigma^{k-1}j_{n}} & S^{2n+k+1} \arrow{u}{\alpha}
\end{tikzcd}.
\]
Here $\overline{\alpha}$ is an extension of $\alpha\circ\Sigma^{k-1}j_{n}$, which exists since 
\begin{align*}
    [\alpha]\in\ker d_1^{k,n} &\Rightarrow ([\Sigma^{k-1}H_{n+1},Y]\circ[\Sigma^{k-1}j_{n},Y])[\alpha]=0 \\
    & \Rightarrow \Sigma^{k-1}H_{n+1}\circ\Sigma^{k-1}j_{n}\circ\alpha\text{ is nullhomotopic} \\
    & \Rightarrow \Sigma^{k-1}j_{n}\circ\alpha\text{ is nullhomotopic on the attaching map} \\
    & \text{ of the top cell of }\Sigma^{k-1}\cp^{n+2}.
\end{align*}
Choosing a different extension $\overline{\alpha}$ changes $\alpha_2$ only within $\im d_1^{k,n+1}(Y)$, so we obtain a well-defined class $[\alpha_2]\in E_2^{k-1,n+2}(Y)$.  

Analogously, the differential $d_r^{k,n}(Y):E_r^{k,n}(Y)\to E_r^{k-1,n+r}(Y)$ is given by the diagram
\begin{equation}
\label{eq:diagram-r-th-differential}
    \begin{tikzcd}
        |[minimum width=2cm]| S^{2n+k+2r} \arrow{r}{\Sigma^{k-1}H_{n+r}} \arrow[bend left, dotted]{rr}{\alpha_r} & |[minimum width=2.5cm]| \Sigma^{k-1}\cp^{n+r} \arrow{r}{\overline{\alpha}} & Y \\
        &\Sigma^{k-1}\cp^{n+1} \arrow{u}{\Sigma^{k-1}i_{n+1,n+r}} \arrow{r}{\Sigma^{k-1}j_{n}} & S^{2n+k+1} \arrow{u}{\alpha}
    \end{tikzcd}
\end{equation}
where the extension $\overline{\alpha}$ exists since $[\alpha]\in\ker d_{r-1}^{k,n}(Y)$, and the different choices of this extension change $\alpha_r$ only within $\im d_{r-1}^{k,n+r-1}(Y)$. Here, $i_{m,n}$ denotes the standard embedding of $\cp^m$ in $\cp^n$, $m\leq n$.

The source of $\alpha$ is actually $\Sigma^{k-1}(\cp^{n+1}/\cp^n)$. The map $\overline{\alpha}$ thus corresponds to an extension of $\alpha$ to $\Sigma^{k-1}(\cp^{n+r}/\cp^n)$, and this correspondence induces a bijection between the set of extensions $\overline{\alpha}$ and $\map(\Sigma^{k-1}(\cp^{n+r}/\cp^n),Y)$. We obtain the following alternate description of the differentials. Given an element $[x]\in[\Sigma^{k-1}(\cp^{n+r}/\cp^{n}),Y]$, the differential $d_r^{k,n}(Y)$ maps the homotopy class of the composition 
\begin{equation}
\label{eq:differential-source-bottom-cell}
    S^{2n+k+1}\xrightarrow{\text{bottom cell}}\textstyle\Sigma^{k-1}(\cp^{n+r}/\cp^{n})\xrightarrow{x}Y
\end{equation} 
to the class represented by the composition 
\begin{equation}
\label{eq:differential-target-Hopf-map}
    S^{2n+2r+k}\xrightarrow{\Sigma^{k-1}H_{n+r}}\textstyle\Sigma^{k-1}(\cp^{n+r}/\cp^{n})\xrightarrow{x}Y. 
\end{equation}
The different choices of $x$ that map to the same element in $\pi_{2n+k+1}(Y)$ correspond to different choices of $\overline{\alpha}$ in \eqref{eq:diagram-r-th-differential} and, as already discussed, these are identified in the spectral sequence.

Now for $Y=G$ we get $\pi_{*}(G)\cong\pi_{*}^s$, as is described, for example, in Chapter 3 of \cite{Madsen}, and therefore $d_r^{k,n}(G):\pi_{2n+k+1}^s\to\pi_{2n+2r+k}^s$. These differentials are determined by classes from 
\[
    [\Sigma^{k-1}(\cp^{n+r}/\cp^{n}),G]\cong[\cp^{n+r}/\cp^{n},\Omega^{k-1}G]\cong\pi_s^{-k+1}(\cp^{n+r}/\cp^{n})
\]
as described in the previous paragraph. Thus, the map 
\[
    [\Sigma^{k-1}H_{n+r},G]:[\Sigma^{k-1}(\cp^{n+r}/\cp^{n}),G]\to\pi_{2n+2r+k}(G)
\]
can be regarded as the map induced by $\Sigma^{k-1}H_{n+r}$ on stable cohomotopy, and, as such, is just the precomposition with the stable Hopf map 
\[
    \Sigma^NS^{2n+2r+k} \xrightarrow{\Sigma^N\Sigma^{k-1}H_{n+r}}\Sigma^N\Sigma^{k-1}(\cp^{n+r}/\cp^n)\xrightarrow{x}S^N,\quad N\gg2n+2r+k.
\]
In particular, this characterization of the differentials is independent of $k$. Finally, since the stable groups $\pi_i^s$ are finite for $i>0$, the spectral sequence splits into $p$-primary components. 

\subsection{The Odd-Primary Sequences}
We use the results of Imanishi \cite{Imanishi} on the $p$-primary homotopy type of $\cp^{n+r}/\cp^{n}$ for odd $p$.
\begin{lema}
\label{stuntIman}
    Let $p$ be an odd prime.
    \begin{enumerate}
        \item For $r<p^2-2$, $\cp^{n+r}/\cp^{n-1}$ has the $\mod p$ homotopy type of a wedge $\bigvee_{i=0}^{p-2}X_{n+i}$, where $$X_{n+i}=S^{2(n+i)}\cup e^{2(n+i+(p-1))}\cup\dots\cup e^{2(n+i+t(p-1))},\hspace{5pt} t=\left\lfloor\frac{r-i}{p-1}
        \right\rfloor.$$
        \item The attaching map $S^{2n+2(p-1)-1}\to S^{2n}$ of the second cell of $X_n$ has order
        \begin{align*}
            \begin{cases}
                p &\text{ if }n\not\equiv0\mod p \\
                1 &\text{ if }n\equiv0\mod p.
            \end{cases}
        \end{align*}
        \item The attaching map $S^{2n+4(p-1)-1}\to S^{2n}\cup e^{2(n+(p-1))}$ of the third cell of $X_n$ has order 
        \begin{align*}
            \begin{cases}
                p^2 &\text{ if }n\not\equiv0,1\mod p \\
                1 &\text{ if }n\equiv-2p+1\mod p^2 \\
                p &\text{ otherwise}.
            \end{cases}
        \end{align*}
    \end{enumerate}
\end{lema}
\begin{proof} 
    (1) is a corollary of \cite[Lemma 1.1]{Imanishi} as stated in the paragraph below Lemma 1.1. (2) and (3) follow from \cite[Proposition 2.4]{Imanishi}.
\end{proof}
In the following, we will use the definition of Toda brackets via extensions and coextensions, which we now present. Let $\alpha\in\{Y,Z\}, \beta\in\{X,Y\}$ and $\gamma\in\{W,X\}$ such that $\alpha\circ\beta=0$ and $\beta\circ\gamma=0$, where $\{X,Y\}=\lim_{n\to\infty}[\Sigma^nX,\Sigma^nY]$ is the set of stable homotopy classes of pointed maps. We say that $\underline{\alpha}\in\{C_{\beta},Z\}$ is the extension of $\alpha$ along $\beta$ if it satisfies $\underline{\alpha}\circ i=\alpha$, where $i:Y\hookrightarrow C_{\beta}$ and $C_{\beta}$ is the mapping cone $Y\cup_{\beta}CX$. Similarly, we call $\widetilde{\gamma}\in\{\Sigma W,C_{\beta}\}$ the coextension of $\gamma$ along $\beta$ if $\pi\circ\widetilde{\gamma}=\Sigma\gamma$, where $\pi:C_{\beta}\to\Sigma X$ collapses $Y\subset C_{\beta}$. We denote by $\ext(\alpha,\beta)$ and $\coext(\beta,\gamma)$ the set of all extensions $\underline{\alpha}$ and coextensions $\widetilde{\gamma}$. The Toda bracket $\langle\alpha,\beta,\gamma\rangle$ is then defined as the set of compositions
\[
    \langle\alpha,\beta,\gamma\rangle=\ext(\alpha,\beta)\circ\coext(\beta,\gamma)\subseteq\{\Sigma W,Z\}.
\]
\begin{cor}(Brumfiel \cite[Proposition 5.6]{Brumfiel-(1970))})
\label{oddprimdiffer}
    \begin{enumerate}
        \item For $r<p^2-2$, the differentials ${}_pd_{r}^{k,n}(G)$ in the $p$-primary spectral sequence are zero unless $r=s(p-1)$ for $1\leq s\leq p-1$.
        \item Up to multiplication by a unit in $\Z_p$ we have:
        \begin{align*}
            {}_pd_{p-1}^{k,n}(G)=
            \begin{cases}
                \alpha_1 &\text{ if }n+1\not\equiv0\mod p \\
                0 &\text{ if }n+1\equiv0\mod p.
            \end{cases}
        \end{align*}
        where by $\alpha_1$ we mean composition with the generator $\alpha_1\in{}_p\pi_{2(p-1)-1}^s\cong\Z_p$.
        \item Up to multiplication by a unit, we have:
        \begin{align*}
            {}_pd_{2(p-1)}^{k,n}(G)=
            \begin{cases}
                \langle\cdot,\alpha_1,\alpha_1\rangle &\text{ if }n+1\not\equiv0,1\mod p \\
                \alpha_{2} & \text{if }n+1\equiv1\mod p\text{ and } n+1\not\equiv-2p+1\mod p^2 \\
                0 & \text{if }n+1\equiv-2p+1\mod p^2 \\
                0\text{ or }\alpha_2 &\text{if }n+1\equiv0\mod p.
            \end{cases}
        \end{align*}
        where by $\alpha_2$ we mean composition with the generator $\alpha_2\in{}_p\pi_{4(p-1)-1}^s\cong\Z_p$ and $\langle\cdot,\alpha_1,\alpha_1\rangle$ is the Toda bracket operation.
    \end{enumerate}
\end{cor}
\begin{proof}
    Recall that by \eqref{eq:differential-source-bottom-cell} and \eqref{eq:differential-target-Hopf-map} the differential ${}_pd_r^{k,n}(G)$ maps the image of $[x]\in{}_p[\Sigma^{k-1}\cp^{n+r}/\cp^n,G]$ under the map 
    \begin{equation}
    \label{eq:differential-source-p-primary}
        {}_p[\text{bottom cell},G]:{}_p[\Sigma^{k-1}\cp^{n+r}/\cp^n,G]\to{}_p\pi_{2n+k+1}(G)
    \end{equation}
    to the image of $[x]$ under
    \begin{equation}
    \label{eq:differential-target-p-primary}
        {}_p[\Sigma^{k-1}H_{n+r},G]:{}_p[\Sigma^{k-1}\cp^{n+r}/\cp^n,G]\to{}_p\pi_{2n+2r+k}(G).
    \end{equation}
    Furthermore, as already discussed, this characterization is independent of $k$. We only need to study the case where $k=1$.
    
        \textbf{(1)} The Hopf map $H_{n+r}:S^{2n+2r+1}\to\cp^{n+r}/\cp^{n}$ corresponds to the attaching map of the top cell of $X_{(n+1)+i}$ from (1) in Lemma \ref{stuntIman}, where $i\equiv r\mod(p-1)$. We have 
        \[
            {}_p[\textstyle\cp^{(n+1)+r}/\cp^{n},G]\cong{}_p[X_{(n+1)},G]\oplus\dots\oplus{}_p[X_{(n+1)+p-2},G].
        \] 
        Since the bottom cell belongs to the subcomplex $X_{(n+1)}$, the subgroup 
        \[
            {}_p[X_{(n+1)+1},G]\oplus\dots\oplus{}_p[X_{(n+1)+p-2},G]
        \]
        maps to zero under \eqref{eq:differential-source-p-primary}. Analogously, the subgroup 
        \[
            \bigoplus_{\substack{0\leq j\leq p-2 \\ j\neq i}}{}_p[X_{(n+1)+j},G]
        \]
        maps to zero under \eqref{eq:differential-target-p-primary}. Thus, ${}_pd_r^{k,n}(G)$ can be non-zero only when $X_{(n+1)+i}=X_{(n+1)}$, i.e. when $r\equiv0\mod(p-1)$.
    
        \textbf{(2)} From (1) in Lemma \ref{stuntIman} it follows that $\cp^{n+(p-1)}/\cp^n\simeq_{p} \bigvee_{i=0}^{p-2}S^{2(n+1)+2i}$. The (stable) Hopf map $\Sigma^NH_{n+(p-1)}$ represents the element $[H_{n+(p-1)}]$ in
        \begin{align*}
            {}_p\pi_{2n+2p-1}^s(\textstyle\cp^{n+p-1}/\cp^{n})&\cong{}_p\pi_{2n+2p-1}^s(S^{2(n+1)})\oplus\dots\oplus{}_p\pi_{2n+2p-1}^s(S^{2(n+1)+2(p-2)}) \\
            &\cong{}_p\pi_{2p-3}^s\oplus\dots\oplus{}_p\pi_{2p-3-2(p-2)}^s,    
        \end{align*}
        where the isomorphism on the first summand is given by $\pi_{2n+2p-1}^s(\text{bottom cell})$, so the generator of this summand is represented by the composition 
        \[
            S^{N+2n+2p-1}\xrightarrow{\alpha_1}S^{N+2(n+1)}\xrightarrow{\text{bottom cell}}\Sigma^N\cp^{n+p-1}/\cp^n.
        \]
        Now, from (2) in Lemma \ref{stuntIman} it follows that 
        \begin{align*}
            [H_{n+(p-1)}]= 
            \begin{cases}
                \pi_{2n+2p-1}^s(\text{bottom cell})(\alpha_1)=[\text{bottom cell}\circ\alpha_1] & \text{if $n+1\not\equiv0\mod p$} \\
                0 & \text{if $n+1\equiv0\mod p$}
            \end{cases}
        \end{align*}
        The differential ${}_pd_{p-1}^{k,n}(G)$ therefore maps the composition $$S^{2n+2}\xrightarrow{\text{bottom cell}}\textstyle\cp^{n+p-1}/\cp^{n}\xrightarrow{x}G$$ either to $0$, or to the composition $$S^{2n+2p-1}\xrightarrow{\alpha_1}S^{2n+2}\xrightarrow{\text{bottom cell}}\textstyle\cp^{n+p-1}/\cp^{n}\xrightarrow{x}G.$$
    
        \textbf{(3)} Again, from (1) in Lemma \ref{stuntIman} it follows that 
        \begin{align*}
            {}_p\pi_{2n+4p-3}^s(\textstyle\cp^{n+2(p-1)}/\cp^{n})&\cong{}_p\pi_{2n+4p-3}^s(X_{(n+1)})\oplus\dots\oplus{}_p\pi_{2n+4p- 3}^s(X_{(n+1)+p-2})
        \end{align*} 
        where (using results from the proof of (2) above)
        \begin{align*}
            {}_p\pi_{2n+4p-3}^s(X_{(n+1)})\cong
            \begin{cases}
                {}_p\pi_{2n+4p-3}^s(S^{2n+2}\bigcup_{\alpha_1}e^{2n+2p}) & \text{if $n+1\not\equiv0\mod p$} \\
                {}_p\pi_{2n+4p-3}^s(S^{2n+2}\vee S^{2n+2p}) &   \text{if $n+1\equiv0\mod p$} 
            \end{cases}
        \end{align*}    
        First suppose $n+1\equiv0\mod p$. Then by \cite[Thm. 4.15]{toda1959p}
        \[
            {}_p\pi_{2n+4p-3}^s(X_{(n+1)})\cong{}_p\pi_{4p-5}^s\oplus{}_p\pi_{2p-3}^s\cong\textstyle\Z_p\{\alpha_2\}\oplus\Z_p\{\alpha_1\}.
        \]
        The isomorphism is given on the first summand by $\pi_{2n+4p-3}^s(\text{bottom cell})$, so the generator of this summand is $[\text{bottom cell}\circ\alpha_2]$. From $(3)$ in Lemma~\ref{stuntIman} we know that the Hopf map represents an element of order $p$ in this group. Thus, on the first summand $[H_{n+2(p-1)}]$ is given by $[\text{bottom cell}\circ\alpha_2]$ or $0$.
        
        In the remaining cases where $n+1\not\equiv0\mod p$, it follows from \cite[Prop.~4.21]{toda1959p} that ${}_p\pi_{2n+4p-3}^s(S^{2n+2}\bigcup_{\alpha_1}e^{2n+2p})\cong\Z_{p^2}$ generated by $\widetilde{\alpha_1}\in\coext(\alpha_1,\alpha_1)$. Since 
        \[
            (\id:S^{2n+2}\cup_{\alpha_1}e^{2n+2p} \to S^{2n+2}\cup_{\alpha_1}e^{2n+2p})\in\ext(\text{bottom cell},\alpha_1),
        \]
        $\widetilde{\alpha_1}\in\langle\text{bottom cell},\alpha_1,\alpha_1\rangle$ and we get from (3) in Lemma~\ref{stuntIman} the following:
        \begin{align*}
            [H_{n+2(p-1)}]\in 
            \begin{cases}
                \langle\text{bottom cell},\alpha_1,\alpha_1\rangle & \text{if $n+1\not\equiv1\mod p$} \\
                0 & \text{if $n+1\equiv-2p+1\mod p^2$} \\
                p\langle\text{bottom cell},\alpha_1,\alpha_1\rangle & \text{otherwise}
            \end{cases}
        \end{align*}
        Finally, we observe that  $\langle p\cdot\id,\alpha_1,\alpha_1\rangle=\alpha_2$ by \cite[13.4]{toda1962composition}. 
\end{proof}

\subsection{The $2$-Primary Sequence}
\begin{lema}(Brumfiel \cite[Proposition 5.8]{Brumfiel-(1970))})
\label{2primdiffer}
    Up to multiplication by a unit, the $2$-primary differentials on the first four pages are given by:
    \begin{align*}
        {}_2d_1^{k,n}(G)&=
        \begin{cases}
            \eta & \text{if $n$ is even} \\
            0 & \text{if $n$ is odd}
        \end{cases}\\ 
        \textit{where }&\eta\in{}_2\pi_1^s\textit{ is the generator,} \\
        {}_2d_2^{k,n}(G)&=
        \begin{cases}
            \nu & \text{if $n\equiv1,2,5,6\mod8$} \\
            2\nu & \text{if $n\equiv0,3\mod8$} \\
            0 & \text{if $n\equiv4,7\mod8$}
        \end{cases}\\
        \textit{where }&\nu\in{}_2\pi_3^s\textit{ is the generator,} \\
        {}_2d_3^{k,n}(G)&=
        \begin{cases}
            \nu\eta^{-1}\nu & \text{if $n\equiv1,5\mod8$} \\
            \langle\cdot,\eta,\nu\rangle & \text{if $n\equiv4\mod8$} \\
            \langle\cdot,\nu,\eta\rangle & \text{if $n\equiv6\mod8$} \\
            \langle\cdot,\eta,\nu\rangle + \langle\cdot,2\nu,\eta\rangle & \text{if $n\equiv0\mod8$} \\
            2\langle\cdot,\eta,\nu\rangle + \langle\cdot,\nu,\eta\rangle & \text{if $n\equiv2\mod8$} \\
            0 & \text{if $n\equiv3,7\mod8$}
        \end{cases} \\
        {}_2d_4^{k,n}(G)&=
        \begin{cases}
            \sigma + \langle\cdot,2\nu,\nu\rangle & \text{if $n\equiv3\mod8$} \\
            \sigma & \text{if $n\equiv4\mod8$} \\
            \sigma + \langle\cdot,\nu\eta^{-1}\nu,\eta\rangle & \text{if $n\equiv5\mod8$} \\
            \sigma + \langle\cdot,\nu,2\nu\rangle & \text{if $n\equiv6\mod8$} \\
            2^{j_0}\sigma & \text{if $n\equiv7\mod8$} \\
            2^{j_1}\sigma + \langle\cdot,2\nu,\nu\rangle & \text{if $n\equiv0\mod8$} \\
            2^{j_2}\sigma + \langle\cdot,\nu,2\nu\rangle + \langle\cdot,\nu\eta^{-1}\nu,\eta\rangle & \text{if $n\equiv1\mod8$} \\
            2^{j_3}\sigma & \text{if $n\equiv2\mod8$}
        \end{cases} \\
        \textit{where }&\sigma\in{}_2\pi_7^s\textit{ is the generator, and }j_0,j_1,j_2,j_3>0\textit{ and depend on }n\mod64.
    \end{align*}
\end{lema}
\begin{proof}
    The proof proceeds similarly to the proof of (2) and (3) in Corollary~\ref{oddprimdiffer}. Since we do not have a version of Lemma~\ref{stuntIman} for the 2-primary case, we determine the relevant$\mod 2$ homotopy types inductively as we go through the proof.
    
        \noindent\textbf{(1) 1st page:} \\
        Firstly, $\cp^{n+1}/\cp^n\simeq S^{2n+2}$ and therefore 
        \[
            {}_2\pi_{2n+3}^s(\cp^{n+1}/\cp^n)\cong {}_2\pi_1^s\cong\Z_2
        \]
        This isomorphism is given by the map $\pi_{2n+3}^s(\text{bottom cell})$, so the generator is $[\text{bottom cell}\circ\eta]$. Now, the stable order of the attaching map of the top cell of $\cp^{n+2}/\cp^n$, that is, the order of 
        \[
            \Sigma^NH_{n+1}:S^{N+2n+3}\to \Sigma^N\cp^{n+1}/\cp^n,
        \]
        is given by the James number $U(n+3,2)$; see e.g. \cite[\S6]{imaoka1984stable}. A table of prime factors of relevant James numbers is given in \cite{lundell} and the relevant $2$-factors are summarized in Table \ref{tab:james-numbers}. Since $U(m,2)_2=1$ if $m$ is even and $2$ if $m$ is odd, we get
        \begin{align*}
            [H_{n+1}]=
            \begin{cases}
                0 & \text{if }n\text{ is odd} \\
                [\text{bottom cell}\circ\eta] & \text{if }n\text{ is even}
            \end{cases}
        \end{align*}
        So the differential ${}_2d_1^{k,n}(G)$ maps the composition $$S^{2n+2}\xrightarrow{\text{bottom cell}}\textstyle\cp^{n+1}/\cp^{n}\xrightarrow{x}G$$ either to $0$, or to the composition $$S^{2n+3}\xrightarrow{\eta}S^{2n+2}\xrightarrow{\text{bottom cell}}\textstyle\cp^{n+1}/\cp^{n}\xrightarrow{x}G.$$
    
        \noindent\textbf{(2) 2nd page:} \\
        From now on, we present the proof only for $n\equiv3,6\mod8$. The other cases follow analogously. From the results in (1) it follows that 
        \begin{align*}
            \cp^{n+2}/\cp^{n}\simeq_2
            \begin{cases}
                S^{2n+2}\vee S^{2n+4} &\text{if $n$ is odd} \\
                S^{2n+2}\cup_\eta e^{2n+4} &\text{if $n$ is even}  
            \end{cases}
        \end{align*}
        and from the corresponding exact sequences we obtain
        \begin{align*}
            \textstyle{}_2\pi_{2n+5}^s(\cp^{n+2}/\cp^{n}) \cong
            \begin{cases}
                {}_2\pi_3^s\oplus{}_2\pi_1^s\cong \Z_8\oplus\Z_2 & \text{if }n\text{ is odd} \\
                {}_2\pi_3^s/\{4\nu\}\cong\Z_4 & \text{if }n\text{ is even} 
            \end{cases}
        \end{align*}
        For even $n$, as well as on the first summand for odd $n$, the isomorphism is given by $\pi_{2n+5}^s(\text{bottom cell})$. Thus, the generator of the corresponding subgroup is $[\text{bottom cell}\circ\nu]$. The isomorphism on the second summand for odd $n$ is given by $\pi_{2n+5}^s(\text{top cell})$, the generator is $[\text{top cell}\circ\eta]$. 
        
        The order of $[H_{n+2}]$ is given by the James number $U(n+4,3)$. From Table~\ref{tab:james-numbers} we see that $U(m,3)_2=4$ if $m\equiv2,7\mod8$, so we obtain
        \begin{align*}
            [H_{n+2}]=
            \begin{cases}
                (2[\text{bottom cell}\circ\nu],y)=([\text{bottom cell}\circ2\nu],y) & \text{if }n\equiv3\mod8\\
                [\text{bottom cell}\circ\nu] & \text{if }n\equiv6\mod8
            \end{cases}
        \end{align*}
        where $y$ is some (yet undetermined) element in the second summand. This proves the statement about the 2nd page differentials.
        
        In addition, we determine the element $y$. Note that after collapsing the bottom cell, the map $H_{n+2}$ is the attaching map 
        \[
            S^{2n+5}\xrightarrow{H_{n+2}}\cp^{n+2}/\cp^{n+1}.
        \]
        From the results in (1) it follows that $y=[\text{top cell}\circ\eta]$.
    
        \noindent\textbf{(3) 3rd page:} \\
        From the results of (2) it follows that
        \[
            \cp^{n+3}/\cp^{n}\simeq_2
            \begin{cases}
                (S^{2n+2}\vee S^{2n+4})\cup_{(2\nu,\eta)}e^{2n+6} & \text{if }n\equiv3\mod8 \\
                S^{2n+2}\cup_{(\eta\text{ }\sqcup\text{ }\nu)}(e^{2n+4}\sqcup e^{2n+6}) & \text{if }n\equiv6\mod8
            \end{cases}
        \]
        and from the corresponding exact sequences of 
        \[
            (K,L):=(\cp^{n+3}/\cp^{n},\cp^{n+2}/\cp^{n})
        \]
        we obtain
        \begin{align*}
            \textstyle{}_2\pi_{2n+7}^s(K)\cong
            \begin{cases}
                {}_2\pi_3^s/\{4\nu\}\cong\Z_4 & \text{if }n\equiv3\mod8 \\
                {}_2\pi_3^s\oplus{}_2\pi_1^s\cong \Z_8\oplus\Z_2 & \text{if }n\equiv6\mod8
            \end{cases}
        \end{align*}
        
        The order of $[H_{n+3}]$ is given by $U(n+5,4)$. From Table~\ref{tab:james-numbers}, $U(m,4)_2=1$ if $m\equiv0\mod8$ and $U(m,4)_2=2$ if $m\equiv3\mod8$. Thus $[H_{n+3}]$ has order $2$ if $n\equiv6\mod8$, and is nullhomotopic if $n\equiv3\mod8$ (which proves the statement of the lemma in this case). 

        It remains to deal with the case where $n\equiv6\mod8$. We have $L\simeq_2S^{2n+2}\cup_{\eta}e^{2n+4}$ and $K/L\simeq S^{2n+6}$. It is not difficult to deduce that ${}_2\pi_{2n+7}^s(L)\cong\Z_8$ is generated by a coextension $\widetilde{\nu}\in\coext(\eta,\nu)$. In the long exact sequence of the pair $(K,L)$: 
        \begin{equation}
        \label{eq:long-exact-seq-K-L}
            \dots\to{}_2\pi_{2n+7}^s(L) \xrightarrow{\pi_{2n+7}^s(i)}{}_2\pi_{2n+7}^s(K)\xrightarrow{\pi_{2n+7}^s(q)}{}_2\pi_{1}^s\to\dots
        \end{equation}
        the map $\pi_{2n+7}^s(i)$ is an extension $\underline{\text{bottom cell}}\in\text{Ext}(\text{bottom cell},\eta)$. Since ${}_2\pi_{2n+7}^s(L)$ maps isomorphically onto the first summand of ${}_2\pi_{2n+7}^s(K)$, this summand is generated by 
        \[
            \underline{\text{bottom cell}}\circ\widetilde{\nu}\in\text{Ext}(\text{bottom cell},\eta)\circ\text{Coext}(\eta,\nu)=\langle\text{bottom cell},\eta,\nu\rangle.
        \]

        The generator of the second summand is found analogously, this time using the subcomplex $L'=S^{2n+2}\cup_{\nu}e^{2n+6}$ (it is a subcomplex since the attaching map of the top cell of $K$ retracts to $S^{2n+2}$). This generator belongs to $\langle\text{bottom cell},\nu,\eta\rangle$. The indeterminacy of both Toda brackets is zero.

        Lastly, we need to express the element $[H_{n+3}]$ of order $2$ in terms of generators. By collapsing the subcomplex $L$, $H_{n+3}$ becomes the attaching map $H_{(n+2)+1}$ of the top cell of $\cp^{n+4}/\cp^{n+2}$, which has order $2$ by results from (1). In other words, in \eqref{eq:long-exact-seq-K-L} we have 
        \[
            {}_2\pi_{2n+7}^s(q)[H_{n+3}]=\eta.
        \]
        Thus, $[H_{n+3}]$ is given by $\langle\text{bottom cell},\nu,\eta\rangle$ on the second summand. Similarly, we have $K/S^{2n+2}\simeq\cp^{n+3}/\cp^{n+1}$ and $H_{n+3}$ becomes the attaching map $H_{(n+1)+2}$ of the top cell of $\cp^{n+4}/\cp^{n+1}$. This map is nullhomotopic, since $U(m,3)_2=1$ for $m\equiv0\mod8$. On the first summand, $[H_{n+3}]$ is therefore zero. To sum up: 
        \begin{align*}
            [H_{n+3}]=
            \begin{cases}
                0 & \text{if }n\equiv3\mod8\\
                (0,\langle\text{bottom cell},\nu,\eta\rangle) & \text{if }n\equiv6\mod8
            \end{cases}
        \end{align*}

        \noindent\textbf{(4) 4th page:} \\
        From the results of (3) we get 
        \[
            \cp^{n+4}/\cp^{n}\simeq_2
            \begin{cases}
                [(S^{2n+2}\vee S^{2n+4})\cup_{(2\nu,\eta)}e^{2n+6}]\vee S^{2n+8} & \text{if }n\equiv3\mod8 \\
                [S^{2n+2}\cup_{(\eta\text{ }\sqcup\text{ }\nu)}(e^{2n+4}\sqcup e^{2n+6})]\cup_{\phi}e^{2n+8} & \text{if }n\equiv6\mod8
            \end{cases}
        \]
        where the attaching map $\phi$ is nullhomotopic on $S^{2n+2}\cup_{\eta}e^{2n+4}$. From the long exact sequences of
        \[
            (K,L):=(\cp^{n+4}/\cp^{n},\cp^{n+3}/\cp^{n})
        \]
        we obtain
        \begin{align*}
            \textstyle{}_2\pi_{2n+9}^s(K)\cong
            \begin{cases}
                {}_2\pi_7^s\oplus{}_2\pi_{3}^s\oplus{}_2 \pi_1^s\cong\Z_{16}\oplus\Z_8\oplus\Z_2 & \text{if }n\equiv3\mod8 \\{}_2\pi_7^s\oplus(2\cdot{}_2\pi_3^s/4\nu)\cong\Z_{16}\oplus\Z_2 & \text{if }n\equiv6\mod8
            \end{cases}
        \end{align*}
        The generator of the first summand is in both cases $[\text{bottom cell}\circ\sigma]$. The generator of third summand for $n\equiv3\mod8$ is clearly $[\text{top cell}\circ\eta]$. For $n\equiv6\mod8$, the generator of the second summand is $\langle\text{bottom cell},\nu,2\nu\rangle$. This follows (similarly as in (3)) from the exact sequence of $(K,L')$, where $L':=S^{2n+2}\cup_{\nu}e^{2n+6}$ is a subcomplex. Finally, to obtain the generator of the second summand for $n\equiv3\mod8$ we first quotient out the $(2n+4)$-cell and get $K_0:=K/S^{2n+4}$, where 
        \[
            {}_2\pi_{2n+9}^s(K_0)\cong{}_2\pi_{2n+9}^s(K).
        \]
        The generator of the second summand of ${}_2\pi_{2n+9}^s(K_0)$ is $\langle\text{bottom cell},2\nu,\nu\rangle$. This follows from the exact sequence of $(K_0,L'')$, where 
        \[
            L'':=S^{2n+2}\cup_{2\nu}e^{2n+6}
        \] 
        is a subcomplex.
    
        The order of $[H_{n+4}]$ is given by $U(n+6,5)$. From Table~\ref{tab:james-numbers}, $U(m,5)_2=16$ if $m\equiv1,4\mod8$, so $H_{n+4}$ has order $16$ if $n\equiv3,6\mod8$ and, in both cases, is given on the first summand by $[\text{bottom cell}\circ\sigma]$. 

        In order to express $H_{n+4}$ on the second summand by generators, we define $L_0:=\cp^{n+2}/\cp^n$. Then $K/L_0\simeq\cp^{n+4}/\cp^{n+2}$ and $H_{n+4}$ becomes the attaching map $H_{(n+2)+2}$ of the top cell of $\cp^{n+5}/\cp^{n+2}$. From results in (2) it follows that this map has order $8$ if $n\equiv3\mod8$, and $2$ if $n\equiv6\mod8$. From analogous arguments with the subcomplex $L_0':=\cp^{n+3}/\cp^n$ it follows that for $n\equiv3\mod8$ the restriction of $[H_{n+4}]$ to the third summand has order $2$. Thus, we have in ${}_2\pi_{2n+9}^s(K)$:
        \begin{align*}
            [H_{n+4}]=
            \begin{cases}
                ([\text{bottom cell}\circ\sigma],\langle\text{bottom cell},2\nu,\nu\rangle,[\text{top cell}\circ\eta]) & \text{if }n\equiv3\mod8\\
                ([\text{bottom cell}\circ\sigma] ,\langle\text{bottom cell},\nu,2\nu\rangle) & \text{if }n\equiv6\mod8
            \end{cases}
        \end{align*}
        For $n\equiv3\mod8$ the top cell, corresponding to the third summand of ${}_2\pi_{2n+9}^s(K)$, splits off. The value of $[H_{n+4}]$ on this summand has therefore no impact on the differential (by arguments similar to the proof of (1) in Corollary~\ref{oddprimdiffer}). Finally, the indeterminacies of the Toda brackets are precisely the $\Z_{16}$-summands.
\end{proof}

\subsection{Results from the Spectral Sequences}
\label{subsect:results-spectral-sequence}
At this point we have a full description of the first four pages of the $p$-primary spectral sequences for $[\Sigma^k\cp^n,G]$. We now discuss how to use these to compute $\coker[\Sigma^k\cp^n,\Omega J]$.

Let $\coker J_p$ denote the homotopy fiber of the map $\beta_p:G/O_{(p)}\to BSO_{(p)}$ defined in \cite[5.16]{Madsen}, and $J_p$ the homotopy fiber of $\Psi^q-\id:BSO_{(p)}\to BSO_{(p)}$, where $q^{p-1}\not\equiv1\mod p^2$ and $\Psi^q$ is the map of classifying spaces induced by the $q$-th Adams operation; see \cite[Chapter 5]{Adams_vect}. It is a result of Sullivan that for each prime $p$ we have splittings
\begin{align}
\label{Sullivansplit1}
    G/O_{(p)}&\simeq BSO_{(p)}\times\coker J_p, \\
\label{Sullivansplit2}
    SG_{(p)}&\simeq J_p\times\coker J_p;
\end{align}
see \cite[Thm.~5.18]{Madsen}. Note that the inclusions
\[
    SG\hookrightarrow G, \quad BSO\hookrightarrow BO,
\]
as well as their localized versions, induce isomorphisms on $\pi_i(\blank)$ for $i>0$. We obtain a splitting of the $p$-primary stable homotopy groups of spheres as 
\begin{equation}
\label{eq:stable-homot-groups-splitting-Sullivan}
    {}_p\pi_i^s\cong\pi_i(SG_{(p)})\cong \pi_i(J_p)\oplus\pi_i(\coker J_p),\quad i>0.
\end{equation}
We need to study the connection between the homotopy groups at the right and the image, resp. cokernel, of the classical $J$-homomorphism. 

A specific solution $\alpha_p:G/O_{(p)}\to BSO_{(p)}$ of the Adams conjecture gives us a diagram of (homotopy) fibrations
\begin{equation}
\label{eq:Adams-conj-diagram-of-fibrations}
    \begin{tikzcd}
        SO_{(p)} \arrow{r}{} \arrow[equal]{d} & J_{p} \arrow{r}{} \arrow[dotted]{d} & BSO_{(p)} \arrow{r}{\Psi^q-\id} \arrow{d}{\alpha_p} & BSO_{(p)} \arrow[equal]{d} \\
        SO_{(p)} \arrow{r}{\Omega J_{(p)}} & SG_{(p)} \arrow{r}{} & G/O_{(p)} \arrow{r}{r_{p}} & BSO_{(p)}
    \end{tikzcd}
\end{equation}
where $r_p$ is the natural inclusion and the dotted map is the first factor of \eqref{Sullivansplit2}; see \cite[Prop 7.7]{brumfiel_transfer} and the proof of \cite[Thm 5.18]{Madsen}. From the induced long exact sequences of homotopy groups we obtain
\begin{equation}
\label{eq:image-J-coker-psi}
    \coker\pi_{i+1}(\Psi^q-\id)\cong\im\pi_i(\Omega J_{(p)}).
\end{equation}
Now, combining \cite[Thm 1.1]{quillen1971adams} with \cite[Thm 1.1, 1.3, 1.5]{adamsIV} we see that $\pi_i(\Psi^q-\id)$ is monomorphic when $i\equiv0\mod4$, and is the zero homomorphism otherwise. From Bott periodicity we get:
\begin{align}
\label{eq:homot-groups-Jp-coker-psi}
    \pi_i(J_p)\cong
    \begin{cases}
        \coker\pi_{i+1}(\Psi^q-\id)\oplus\pi_i(BSO_{(p)}) & \text{if }i\equiv1,2\mod8 \\
        \coker\pi_{i+1}(\Psi^q-\id) & \text{otherwise}
    \end{cases}
\end{align}
From \eqref{eq:image-J-coker-psi} and \eqref{eq:homot-groups-Jp-coker-psi} we have: 
\begin{align}
\label{eq:homot-groups-Jp-image-of-J}
    \pi_i(J_p)\cong
    \begin{cases}
        \im\pi_i(\Omega J_{(p)})\oplus\Z_2 & \text{if }p=2\text{ and }i\equiv1,2\mod8 \\
        \im\pi_i(\Omega J_{(p)}) & \text{otherwise}
    \end{cases}
\end{align}
The additional $\Z_2$-summands are known to be generated by the Adams elements $\mu_{r}\in\pi_r^s$ from \cite[Thm 1.2]{adamsIV}; see also the text below \cite[Conjecture~5]{article}. It follows from \eqref{Sullivansplit2} and \eqref{eq:homot-groups-Jp-image-of-J} that 
\begin{align}
\label{eq:homot-groups-cokerJp-cokernel-of-J}
    \pi_i(\coker J_p)\cong
    \begin{cases}
        \coker(\pi_i(\Omega J_{(p)}))/\{\mu_i\} & \text{if }p=2 \text{ and } i\equiv1,2\mod8 \\
        \coker(\pi_i(\Omega J_{(p)})) & \text{otherwise}
    \end{cases}
\end{align}

Now, from (\ref{Sullivansplit1}) we have: 
\[
    \textstyle[\Sigma^k\cp^n,G/O]_{(p)}\cong [\Sigma^k\cp^n,BSO]_{(p)}\oplus[\Sigma^k\cp^n,\coker J_p].
\]
Recall that in \cite[Lemma 4.1]{kaluzny2026highersmoothsurgerystructure} we obtained (after localization) the splitting 
\[
    \textstyle[\Sigma^k\cp^n,G/O]_{(p)}\cong \ker[\Sigma^k\cp^n,J]_{(p)}\oplus\coker[\Sigma^k\cp^n,\Omega J]_{(p)}.
\]
Note that $\coker[\Sigma^k\cp^n,\Omega J]$ is finite so $\coker[\Sigma^k\cp^n,\Omega J]_{(p)}\cong{}_p\coker[\Sigma^k\cp^n,\Omega J]$. By comparing $\ker[\Sigma^k\cp^n,J]_{(p)}$ with $[\Sigma^k\cp^n,BSO]_{(p)}$ we find that
\begin{align}
\label{eq:coker-Jp-CP}
    \textstyle{}_p\coker[\Sigma^k\cp^n,\Omega J]\cong
    \begin{cases}
        [\Sigma^k\cp^n,\coker J_p]\oplus\Z_2 & \text{if }p=2\text{ and }(k\mod8,n\mod4) \\ 
        & \in\{(0,1),(3,3),(4,3),(7,1)\}\subset\Z_8\times\Z_4 \\ 
        [\Sigma^k\cp^n,\coker J_p] & \text{otherwise}
    \end{cases}
\end{align}
The generators of the additional $\Z_2$-summands are $[\Sigma^kj_{n-1},G/O_{(2)}](\mu_{2n+k})$, where $\mu_{2n+k}$ is the Adams element discussed above viewed as a class in $\pi_{2n+k}(G/O_{(2)})$. 

We can take $Y=\coker J_p$ and obtain a spectral sequence for $[\Sigma^k\cp^n,\coker J_p]$ as described at the beginning of this chapter up to Lemma \ref{sesspec}. From (\ref{Sullivansplit2}) we obtain a map of spectral sequences ${}_pE_*^{*,*}(G)\Rightarrow E_*^{*,*}(\coker J_p)$ and this map is surjective on the first four pages. Thus, we can compute the differentials of $E_r^{k,n}(\coker J_p)$ for $1\leq r\leq4$ from Corollary \ref{oddprimdiffer} and Lemma \ref{2primdiffer}. These results are summarized in \ref{subsec:sprectral-sequence-coker-J2} and \ref{subsec:sprectral-sequence-coker-J3} for $p=2,3$. For higher $p$, all homotopy groups of $\coker J_p$ are zero in the relevant range.
\begin{lema}
\label{lema:coker-J-CP-values}
    The values of $\coker[\Sigma^{k}\cp^n,\Omega J]\subseteq[\Sigma^{k}\cp^n,G/O]$, as well as its generators, are summarized for $1\leq n \leq6$ and $0\leq k\leq6$ in the following table:
    \begin{table}[ht!]
        \centering
        \begin{tabular}{|c|c|c|c|c|} \hline
            \diaghead{\theadfont aaaaaaaa}%
            {$k$}{$n$} & $1$ & $2$ & $3$ & $4$ \\ 
            \hline
            $0$ & \thead{$\Z_2$ \\ $\{\eta^2\}$} & $\cdot$ & \thead{$\Z_2$ \\ $\{\nu^2\}$} & \thead{$\Z_4$ \\ $\{\underline{\nu^2}\}$} \\
            \hline
            $1$ & $\cdot$ & $\cdot$ & $\cdot$ & $\cdot$ \\
            \hline
            $2$ & $\cdot$ & \thead{$\Z_2$ \\ $\{\nu^2\}$} & \thead{$\Z_2^{2}$ \\ $\{\underline{\nu^2},\overline{\nu}\}$} & \thead{$\Z_2\oplus\Z_3$ \\ $\{\underline{\nu^2}_2,\beta_1\}$} \\
            \hline
            $3$ & $\cdot$ & $\cdot$ & \thead{$\Z_2^2$ \\ $\{\eta\overline{\nu},\mu_{9}\}$} & \thead{$\Z_2$ \\ $\{\underline{\eta\overline{\nu}}\}$} \\
            \hline
            $4$ & \thead{$\Z_2$ \\ $\{\nu^2\}$} & \thead{$\Z_4$ \\ $\{\underline{\nu^2}\}$} & \thead{$\Z_4\oplus\Z_2\oplus\Z_3$ \\ $\{\underline{\nu^2}_2,\mu_{10},\beta_1\}$} & \thead{$\Z_4\oplus\Z_3$ \\ $\{\underline{\nu^2}_3,\underline{\beta_1}\}$} \\
            \hline
            $5$ & $\cdot$ & $\cdot$ & $\cdot$ & $\cdot$ \\
            \hline
            $6$ & \thead{$\Z_2$ \\ $\{\overline{\nu}\}$} & \thead{$\Z_3$ \\ $\{\beta_1\}$} & \thead{$\Z_3$ \\ $\{\underline{\beta_1}\}$} & \thead{$\Z_2^2$ \\ $\{\sigma^2,\kappa\}$} \\
            \hline
            \diaghead{\theadfont aaaaaaaa}%
            {$k$}{$n$} & \multicolumn{2}{c|}{$5$} & \multicolumn{2}{c|}{$6$} \\ \hline
            $0$ & \multicolumn{2}{c|}{\thead{$\Z_2^{2}\oplus\Z_3$ \\ $\{\underline{\overline{\nu}},\mu_{10},\beta_1\}$}} & \multicolumn{2}{c|}{\thead{$\Z_2\oplus\Z_3$ \\ $\{\underline{\overline{\nu}}_2,\underline{\beta_1}\}$}} \\ \hline
            $1$ & \multicolumn{2}{c|}{$\cdot$} & \multicolumn{2}{c|}{$\cdot$} \\ \hline
            $2$ & \multicolumn{2}{c|}{\thead{$\Z_2\oplus\Z_3$ \\ $\{\underline{\nu^2}_3,\underline{\beta_1}\}$}} & \multicolumn{2}{c|}{\thead{$\Z_2^{3}$ \\ $\{\underline{\nu^2}_{4},\sigma^2,\kappa\}$}} \\ \hline
            $3$ & \multicolumn{2}{c|}{\thead{$\Z_2\oplus\Z_3$ \\ $\{\underline{\eta\overline{\nu}}_2,\alpha_1\beta_1\}$}} & \multicolumn{2}{c|}{\thead{$\Z_2\oplus\Z_3$ \\ $\{\underline{\eta\overline{\nu}}_3,\underline{\alpha_1\beta_1}\}$}} \\ \hline
            $4$ & \multicolumn{2}{c|}{\thead{$\Z_4\oplus\Z_2^{2}\oplus\Z_3$ \\ $\{\underline{\nu^2}_4,\sigma^2,\kappa,\underline{\beta_1}_2\}$}} & \multicolumn{2}{c|}{\thead{$\Z_4^{2}\oplus\Z_3$ \\ $\{\underline{\nu^2}_5,\underline{\sigma^2},\underline{\beta_1}_3\}$}} \\ \hline
            $5$ & \multicolumn{2}{c|}{\thead{$\Z_2$ \\ $\{\eta\kappa\}$}} & \multicolumn{2}{c|}{\thead{$\Z_4$ \\ $\{\underline{\eta\kappa}\}$}} \\ \hline
            $6$ & \multicolumn{2}{c|}{\thead{$\Z_2^{3}$ \\ $\{\underline{\sigma^2},\underline{\kappa},\eta_4\}$}} & \multicolumn{2}{c|}{\thead{$\Z_2^2\oplus\Z_4$ \\ $\{\underline{\sigma^2}_2,\underline{\kappa}_2,\nu_4\}$}} \\ \hline
        \end{tabular}
        \label{tabnormtorsCP}
    \end{table}
\end{lema}
\begin{com}
    We explain the simplified notation used for the generators in this table. First, the notation for elements in $\coker\pi_*(\Omega J)\subseteq\pi_*^s$ is the standard one, see e.g. \cite[p.385]{hatcher2002algebraic}, and an element $x\in\coker\pi_{2n+k}(\Omega J)$ actually denotes the image of $x$ under the collapse map
    \[
        [\Sigma^kj_{n-1},G/O]:\pi_{2n+k}(G/O)\to[\Sigma^k\cp^n,G/O],
    \]
    where we consider $x$ as an element in $\pi_{2n+k}(G/O)$ by an analogue of \eqref{eq:splitting-norm-inv-thom-space} for $S^{2n+k}$. Secondly, underlined elements denote extensions along $\Sigma^kH_{n-1}$; see the paragraph below Lemma~\ref{stuntIman}. Lastly, a subscript next to the underlined element refers to the number of successive extensions taken. 
    
    For instance, the subgroup $\coker[\Sigma^4\cp^3,\Omega J]$ consists of a $\Z_4$-summand generated by an extension of $\underline{\nu^2}\in\ext(\nu^2,\Sigma^4H_1)$ along the map $\Sigma^4H_2$, of a $\Z_2$-summand generated by $[\Sigma^4j_2,G/O](\mu_{10})$, and a $\Z_3$-summand generated by $[\Sigma^4j_2,G/O](\beta_1)$ (or, equivalently, $[\Sigma^4j_2,\coker J_3](\beta_1)$ followed by the inclusion of $\coker J_3$ into $G/O_{(3)}$). 
    
    Note that the choice of particular extensions is not specified. However, this choice will play an important role in the proof of Lemma~\ref{obstrtors} where it is discussed in more detail. 
\end{com}
\begin{proof}
    The values for $k=0$ were already calculated by Brumfiel in \cite{Brumfiel-(1970))}. We present the calculation for the $2$-primary part of the row $k=4$. The other cases follow analogously. 

    From Lemma \ref{sesspec} with $n=1, k=4$ and the spectral sequences in \ref{subsec:sprectral-sequence-coker-J2} and \ref{subsec:sprectral-sequence-coker-J3} we obtain the short exact sequence 
    \[
        \textstyle0\to\Z_2\{\overline{\nu}\}\to[\Sigma^4\cp^2,\coker J_2] \to\pi_6(\coker J_2)\cong\Z_2\{\nu^2\}\to0\to0.
    \]
    This extension problem was solved in the proof of \cite[Lemma 2.5,(iii)]{kasilingam2026diffeomorphismclassificationsmoothstructures}, see the paragraph below sequence (12). The sequence does not split.
    
    We input $[\Sigma^4\cp^2,\coker J_2]\cong\Z_4\{\underline{\nu^2}\}$ into the sequence for $k=4,n=2$:
    \[
        \textstyle0\to0\to[\Sigma^4\cp^3,\coker J_2] \to\Z_4\{\underline{\nu^2}\}\to0\to0.
    \]
    Thus, $[\Sigma^4\cp^3,\coker J_2]\cong\Z_4\{\underline{\nu^2}_2\}$. The sequence for $k=4,n=3$ is also
    \[
        \textstyle0\to0\to[\Sigma^4\cp^4,\coker J_2]\to\Z_4\{\underline{\nu^2}_2\} \to0\to0, 
    \]
    so $[\Sigma^4\cp^4,\coker J_2]\cong\Z_4\{\underline{\nu^2}_3\}$. Next, the sequence for $k=n=4$ is: 
    \[
        \textstyle0\to\Z_2^2\{\sigma^2,\kappa\}\to[\Sigma^4\cp^5,\coker J_2]\to \Z_4\{\underline{\nu^2}_3\}\to0\to0
    \]
    We need to solve this extension problem. The elements $\sigma^2,\kappa$ map in $[\Sigma^4\cp^5,\coker J_2]$ to $[\Sigma^4j_4,\coker J_2](\sigma^2)$ and $[\Sigma^4j_4,\coker J_2](\kappa)$; see Lemma~\ref{sesspec}. Now, the sequence for $k=4,n=5$ is 
    \[
        \textstyle0\to\Z_2\{\eta_4\}\to[\Sigma^4\cp^6,\coker J_2]\to [\Sigma^4\cp^5,\coker J_2]\to\Z_2\{\eta\kappa\}\to0
    \]
    where the map onto the rightmost $\Z_2$-summand is $[\Sigma^4H_5,\coker J_2]$. The composition $[\Sigma^4H_5,\coker J_2]\circ[\Sigma^4j_4,\coker J_2]$ is by definition \eqref{eq:1st-differential-composition} the differential $d_1^{5,4}$. Thus, the image of $\sigma^2$ and $\kappa$ under the composition
    \[
        \Z_2^2\cong\pi_{14}(\coker J_2)\xrightarrow{[\Sigma^4j_4,\coker J_2]}[\Sigma^4\cp^5,\coker J_2]\xrightarrow{[\Sigma^4H_5,\coker J_2]}\pi_{15}(\coker J_2)\cong\Z_2
    \]
    is given by $d_1^{5,4}(\sigma^2)$ and $d_1^{5,4}(\kappa)$. From \ref{subsec:sprectral-sequence-coker-J2} we see that $\Z_2\ni d_1^{5,4}(\kappa)\neq0$, so $[\Sigma^4j_4,\coker J_2](\kappa)$ must generate a $\Z_2$-summand in $[\Sigma^4\cp^5,\coker J_2]$. For the other generator we note that $\Z_2\ni d_3^{5,4}(\sigma^2)\neq0$. Thus, by a similar argument using the sequence for $k=4,n=7$ (and the fact that the generator $[\Sigma^4j_4,\coker J_2](\sigma^2)$ lifts to $[\Sigma^4\cp^7,\coker J_2]$) we conclude that this element also generates a $\Z_2$-summand in $[\Sigma^4\cp^5,\coker J_2]$. In summary, 
    \[
        [\Sigma^4\cp^5,\coker J_2]\cong\Z_4\{\underline{\nu^2}_4\}\oplus\Z_2^2\{[\Sigma^4j_4,\coker J_2](\sigma^2),[\Sigma^4j_4,\coker J_2](\kappa)\}.
    \]

    The sequence for $k=4,n=5$ is then the following:
    \begin{equation}
    \label{eq:spectr-seq-exact-seq-k4-n5}
        \textstyle0\to\Z_2\{\eta_4\}\to[\Sigma^4\cp^6,\coker J_2]\to \Z_4\oplus\Z_2^2\to\Z_2\{\eta\kappa\}\to0.
    \end{equation}
    Once again, we need to solve an extension problem. We use a strategy similar to \cite[Lemma 2.5,(iii)]{kasilingam2026diffeomorphismclassificationsmoothstructures}. We know that the rightmost (non-zero) map is zero on all generators except $[\Sigma^4j_4,\coker J_2](\kappa)$. We need to determine whether $[\Sigma^4j_5,\coker J_2](\eta_4)$ is a multiple of $\underline{\sigma^2}$ or $\underline{\nu^2}_5$. In order to do this, we study the long exact sequence
    \[
        \dots\to\pi_{15}(\coker J_2)\xrightarrow{\eta}\pi_{16}(\coker J_2)\to[\Sigma^4(\cp^6/\cp^4),\coker J_2]\to\pi_{14}(\coker J_2)\xrightarrow{\eta}\dots
    \] 
    induced by the cofibration sequence of $\Sigma^4(\cp^6/\cp^4)\simeq\Sigma^4(S^{10}\cup_{\eta}e^{12})$. We get
    \begin{equation}
    \label{eq:short-exact-seq-stunted-cp6-cp4}
        0\to\Z_2\{\eta_4\}\to[\Sigma^4(\cp^6/\cp^4),\coker J_2]\to\Z_2\{\sigma^2\}\to0.
    \end{equation}
    For an extension $\underline{\sigma^2}$ of $\sigma^2$ (along the attaching map of the top cell) we have:
    \begin{align}
    \label{eq:extension-of-sigma^2}
        2\underline{\sigma^2}&=\underline{\sigma^2}\circ(2\Sigma^{12}\id_{\cp}) \\ \nonumber
        &=\underline{\sigma^2}\circ(\Sigma^{12}i_1\circ\underline{2\id_{S^{14}}}+\widetilde{2\id_{S^{15}}}\circ\Sigma^{12}j_1) \\ \nonumber
        &=\sigma^2\circ\underline{2\id_{S^{14}}} +\underline{\sigma^2}\circ\widetilde{2\id_{S^{15}}}\circ\Sigma^{12}j_1
    \end{align}
    where the second equality comes from \cite[Corollary 2.6,(1)]{Kachi2001SomeCG}. Here 
    \[
        \text{Ext}(\sigma^2,\eta)\circ\text{Coext}(\eta,2\id_{S^{15}})\in\langle\sigma^2,\eta,2\id_{S^{15}}\rangle
    \]
    which is zero; see e.g. \cite[Thm 2.1,ii)]{mukai1969stable}. In addition, from \cite[Proposition 2.7,(2)]{Kachi2001SomeCG} we have
    \[
        \sigma^2\circ\underline{2\id_{S^{14}}} \in\langle\sigma^2,2\id_{S^{14}},\eta\rangle\circ\Sigma^{12}j_1,
    \]
    where the indeterminacy of this Toda bracket is $0$ in $\pi_{16}(\coker J_2)$, and by \cite{hatcher2002algebraic} (the paragraph above the proof of 4.56) it contains $\eta_4$. We substitute in \eqref{eq:extension-of-sigma^2}:
    \begin{align}
        2\underline{\sigma^2}&=\eta_4\circ\Sigma^{12}j_1,
    \end{align}
    where $[\eta_4\circ\Sigma^{12}j_1]=[\Sigma^{12}j_1,\coker J_2](\eta_4)$. Thus, the sequence~\eqref{eq:short-exact-seq-stunted-cp6-cp4} doesn't split and $[\Sigma^4(\cp^6/\cp^4),\coker J_2]\cong\Z_4\{\underline{\sigma^2}\}$. Finally, from the long exact sequence associated to the cofibration $\cp^4\to\cp^6\to\cp^6/\cp^4$ we obtain:
    \[
        \begin{tikzcd}
            0 \arrow{r}{} & {[\Sigma^{4}(\cp^6/\cp^4),\coker J_2]} \arrow{r}{} & {[\Sigma^{4}\cp^6,\coker J_2]} \arrow{r}{} & {[\Sigma^{4}\cp^4,\coker J_2]} \arrow{r}{} & 0 \\
            & \Z_4\{\underline{\sigma^2}\} \arrow[u,phantom,sloped,"\cong"] & & \Z_4\{\underline{\nu^2}_3\} \arrow[u,phantom,sloped,"\cong"] & 
        \end{tikzcd}
    \]
    Together with \eqref{eq:spectr-seq-exact-seq-k4-n5} this implies that $[\Sigma^4\cp^6,\coker J_2]\cong\Z_4^2\{\underline{\sigma^2},\underline{\nu^2}_5\}$.
    
    Finally, from \eqref{eq:coker-Jp-CP} we have
    \begin{align*}
        {}_2\coker[\Sigma^4\cp^n,\Omega J]\cong
        \begin{cases}
            [\Sigma^4\cp^n,\coker J_2] &\text{if }n\not\equiv3\mod4 \\
            [\Sigma^4\cp^n,\coker J_2]\oplus\Z_2 &\text{if }n\equiv3\mod4
        \end{cases}
    \end{align*}
    so for $n=3$ we obtain an additional element $[\Sigma^4j_{3},G/O](\mu_{10})$.
\end{proof}

%% file: obstr.tex
\section{The Surgery Obstruction}
\label{sec:surgery-obstruction}

In this section, we compute the surgery obstruction map 
\[
    \sigma_{n,2l}:\sN_{\partial}(\cp^n\times D^{2l})\to L_{2(n+l)}(\Z)
\]
in cases where $n+l$ is odd. In \cite[Section 3]{kaluzny2026highersmoothsurgerystructure} we claim that the obstruction is given by the Arf invariant of the quadratic form associated with a normal map. We now present a formula for computing this invariant. This formula involves some special characteristic class and can be seen as an analogue of \cite[(5.6)]{kaluzny2026highersmoothsurgerystructure}.

We use the results of Rourke and Sullivan; see \cite{Rourke}, resp. \cite[0.6]{Hambleton}. Let $V_{\cp^n\times D^{2l}}$ denote the total Wu class of $\cp^n\times D^{2l}$. There is a class $$K=k_2+k_6+\dots\in\prod_{i\geq0} \textstyle\widetilde{H}^{4i+2}(G/TOP;\Z_2)$$ such that for any $\TOP$-normal map corresponding to $g:Th(\cp^n\times D^{2l})\to G/TOP$, the associated Arf invariant has value 
\[
    \textstyle\langle V_{\cp^n\times D^{2l}}^2\cup\widetilde{H}^*(g)(K),[\cp^n\times D^{2l},\cp^n\times S^{2l-1}]_{\Z_2}\rangle\in H_0(\cp^n\times D^{2l};\Z_2)\cong\Z_2
\]

The surgery obstruction factors through $G/O\xrightarrow{s}G/TOP$. The $\DIFF$-surgery obstruction of the map $f:Th(\cp^n\times D^{2k})\to G/O$ can thus be computed as the $\TOP$-obstruction of the composition 
\[
    Th(\cp^n\times D^{2k})\xrightarrow{f}G/O\xrightarrow{s}G/TOP
\]
by the formula 
\begin{equation}
\label{eq:Arf-inv-general-formula}
    \textstyle\langle V_{\cp^n\times D^{2l}}^2\cup\widetilde{H}^*(s\circ f)(K),[\cp^n\times D^{2l},\cp^n\times S^{2l-1}]_{\Z_2}\rangle.
\end{equation}

From now on, we denote the smooth class $\widetilde{H}^*(s)(K)$ by $\widetilde{K}$. 

Now, by \cite[Thm 1.3]{brumfiel_PL} (and the paragraph above it) the image of $k_l$ in $\widetilde{H}^*(G;\Z_2)$ under the homomorphism induced by the map $G\to G/TOP$ is zero unless $l=2^j-2$ for $j>1$. This map factors through $G/O$ and we get:
\[
\begin{tikzcd}
      & \widetilde{H}^*(G/O;\Z_2) \arrow{dr}{} & \\
     \widetilde{H}^*(G/TOP;\Z_2) \arrow{rr}{} \arrow{ur}{\widetilde{H}^*(s)} & & \widetilde{H}^*(G;\Z_2) 
\end{tikzcd}
\]
The upper right homomorphism is injective (the dual map on homology is surjective, see e.g. \cite[12.3]{Madsen}), so $\widetilde{k_l}$ is zero in $\widetilde{H}^*(G/O;\Z_2)$ unless $l=2^j-2$ for $j>1$.

The remaining values of $\widetilde{K}$ have been studied by Brumfiel and Madsen; see \cite[\S8]{brumfiel_transfer}. They construct a specific solution of the Adams conjecture localized at $2$, i.e. a splitting 
\begin{equation}
\label{adamsconj}
    \begin{tikzcd}
      & G/O_{(2)} \arrow{dr}{r_2} & \\
     BO \arrow{rr}{\Psi^3-\id} \arrow{ur}{\alpha_2} & & BO_{(2)}
    \end{tikzcd}
\end{equation}
where $\Psi^3$ denotes the Adams operation in the $KO$-ring and $r_2$ is the natural map; see also \eqref{eq:Adams-conj-diagram-of-fibrations}. They then prove that 
\begin{align}
\label{Kvalues}
    \textstyle\widetilde{H}^*(BO;\Z_2)\ni\widetilde{H}^*(\alpha_2)(\widetilde{k_l})=
    \begin{cases}
        \displaystyle\sum_{u<v}\sum_{i=0}^{2^j-2}t_u^it_v^{2^j-2-i} &\text{ if }l=2^j-2 \\
        0 &\text{ otherwise}
    \end{cases}
\end{align}
where we identify $H^*(BO;\Z_2)$ with the ring of symmetric power series in variables $t_i$ in the usual way; see e.g. \cite[\S7]{milnor1974characteristic}. Alternatively, let $s_i\in H^i(BO;\Z_2)$ be the unique non-zero primitive and $\overline{\Delta}$ denote the reduced diagonal in the Hopf algebra $H^*(BO;\Z_2)$. Then $\widetilde{H}^*(\alpha_2)(\widetilde{K})$ is the unique class that satisfies:
\begin{enumerate}
    \item $\overline{\Delta}\widetilde{H}^*(\alpha_2)(\widetilde{K})=\displaystyle\sum_{r\geq2}\sum_{\substack{i+j=2^r-2 \\ i,j\geq1}}s_i\otimes s_j$ \\
    \item $\langle\widetilde{H}^*(\alpha_2)(\widetilde{K}),[\rp^n]_{\Z_2}\rangle=
    \begin{cases}
        1 & \text{ if }n=2^r-2, r\geq2 \\
        0 & \text{ otherwise}
    \end{cases}$
\end{enumerate}
We can translate this in terms of universal Stiefel-Whitney classes $w_i\in H^i(BO;\Z_2)$. We will only need the following statement:
\begin{lema}
\label{KSWclass}
    The class $\widetilde{H}^*(\alpha_2)(\widetilde{k}_{2^r-2})$ is a homogeneous polynomial of degree $2^r-2$ in Stiefel-Whitney classes of the form $$P_{2^r-2}(w_i)=w_{2^r-2}+w_1^{2^r-2}+\sum_{\substack{\alpha=(\alpha_1,\dots,\alpha_{2^r-3})\in\N_0^{2^r-3} \\ \sum_ii\alpha_i=2^r-2 \\ \alpha_1\neq2^r-2}}c_{\alpha}w_1^{\alpha_1}\dots w_{2^r-3}^{\alpha_{2^r-3}},$$where $c_\alpha\in\Z_2$.
\end{lema}
\begin{proof}
    The Stiefel-Whitney classes are the elementary symmetric functions $\sigma_i$ in variables $t_j$ and they generate the ring $H^*(BO;\Z_2)$. By (\ref{Kvalues}), $\widetilde{H}^*(\alpha_2)(\widetilde{k}_{2^r-2})$ contains $t_u^{2^r-2}$ as a summand for all $u\geq1$. That is possible only if it contains $w_1^{2^r-1}$ as a summand. 

    For the other claim, we must first recall that $\overline{\Delta}(\sigma_i)=\sum_{m=1}^{i-1}\sigma_m\otimes\sigma_{i-m}$. The primitive elements in the Hopf algebra of symmetric functions are the power sums $p_n=\sum_{u}t_u^n$. The relation between $\sigma_i$s and $p_n$s is expressed by the Newton identities:
    \[
        p_n=(-1)^{n-1}n\sigma_n+\sum_{i=1}^{n-1}(-1)^{k-1+i}\sigma_{k-i}p_i
    \]
    From the formula for $\overline{\Delta}\widetilde{H}^*(\alpha_2)(\widetilde{K})$, we see that in degree $2^r-2$ it must contain $$p_{2^r-3}\otimes p_1 + p_1\otimes p_{2^r-3}=\sigma_{2^r-3}\otimes\sigma_1 + \sigma_1\otimes\sigma_{2^r-3}+\dots$$ as a summand. But that is possible only if $\widetilde{H}^*(\alpha_2)(\widetilde{K})$ contains $\sigma_{2^r-2}$ as a summand (using the formula for the reduced diagonal of $\sigma_i$).
\end{proof}
\begin{com}
    In low degrees, we can compute explicitly: 
    \[
        \widetilde{H}^*(\alpha_2)(\widetilde{k}_2)=w_1^2+w_2,\quad \widetilde{H}^*(\alpha_2)(\widetilde{k}_6)=w_1^6+w_1^4w_2+w_2^3+w_3w_1^3+w_4w_2+w_6.
    \] 
\end{com}

Now, as presented in \cite[(4.1) and Lemma 4.1]{kaluzny2026highersmoothsurgerystructure}
\begin{align*}
    \textstyle[Th(\cp^n\times D^{2l}),G/O]&\cong\textstyle[\Sigma^{2l}\cp^n,G/O]\oplus\pi_{2l}(G/O) \\    &\cong\ker[\Sigma^{2l}\cp^n,J]\oplus\coker[\Sigma^{2l}\cp^n,\Omega J]\oplus\pi_{2l}(G/O) \\
    [f]&\mapsto([f_{free}],[f_{tors}],[f_{S^{2l}}])
\end{align*}
where $\Z^{t_{n,2l}'}\cong\ker[\Sigma^{2l}\cp^n,J]\subset[\Sigma^{2l}\cp^n,BO]$.
The first isomorphism is induced by a homeomorphism, which we omit from our notation. The second isomorphism is obtained from the fibration sequence $G\xrightarrow{\overline{i}}G/O\xrightarrow{r}BO\xrightarrow{J}BG$. We get from~\eqref{eq:Arf-inv-general-formula}:
\begin{align*}
    \sigma_{n,2l}(f)=\textstyle\langle V_{\cp^n\times D^{2l}}^2\cup [\widetilde{H}^*(p[f_{free}])+\widetilde{H}^*(\overline{i}\circ f_{tors})+\widetilde{H}^*(f_{S^{2l}})](\widetilde{K}),[\cp^n\times D^{2l},\cp^n\times S^{2l-1}]_{\Z_2}\rangle
\end{align*}
where
\[
    p:\ker[\Sigma^{2l}\cp^n,J]\to[\Sigma^{2l}\cp^n,G/O]
\]
is some (so far unspecified) component of the splitting in \cite[Lemma~4.1]{kaluzny2026highersmoothsurgerystructure} and $p([f_{free}])$ here denotes a map $\Sigma^{2k}\cp^n\to G/O$ representing the image of $[f_{free}]$ under this splitting. Since $\sigma_{n,2l}$ is a homomorphism to $\Z_2$, we can localize at $2$. 

Now, we construct a specific localized splitting $p_{(2)}$. From (\ref{adamsconj}) we have:
\[
\begin{tikzcd}
    & \left[\Sigma^{2l}\cp^n,G/O_{(2)}\right] \arrow{dr}{[\Sigma^{2l}\cp^n,r_2]} & \\
    \left[\Sigma^{2l}\cp^n,BO\right] \arrow{rr}{[\Sigma^{2l}\cp^n,\Psi^3-\id]} \arrow{ur}{[\Sigma^{2l}\cp^n,\alpha_2]} & & \left[\Sigma^{2l}\cp^n,BO_{(2)}\right]
\end{tikzcd}
\]
We can extend $[\Sigma^{2l}\cp^n,\alpha_2]$ to $[\Sigma^{2l}\cp^n,\alpha_{(2)}]:[\Sigma^{2l}\cp^n,BO_{(2)}]\to[\Sigma^{2l}\cp^n,G/O_{(2)}]$ and similarly for $[\Sigma^{2l}\cp^n,\Psi^3-\id]_{(2)}$. Note that $\im[\Sigma^{2l}\cp^n,r_2]=\ker[\Sigma^{2l}\cp^n,J]_{(2)}$. 
\begin{lema}
    The map $[\Sigma^{2l}\cp^n,\Psi^3-\id]_{(2)}$ surjects onto $\im[\Sigma^{2l}\cp^n,r_2]$. In addition, it is an isomorphism unless $(l,n)\equiv(0,1)$ or $(4,3)$ $(\mod 8,\mod 4)$ when $\ker[\Sigma^{2l}\cp^n,\Psi^3-\id]_{(2)}\cong\Z_2$. 
\end{lema}
\begin{proof}
    We first prove the statement for $(l,n)\not\equiv(0,1),(4,3)$ $(\mod8,\mod4)$. In these cases, the (localized) diagram above is of the form:
    \[
    \begin{tikzcd}
        & \Z_{(2)}^{t_{n,2l}'}\oplus2\text{-torsion} \arrow{dr}{[\Sigma^{2l}\cp^n,r_2]} & \\
        \Z_{(2)}^{t_{n,2l}'} \arrow{rr}{[\Sigma^{2l}\cp^n,\Psi^3-\id]_{(2)}} \arrow{ur}{[\Sigma^{2l}\cp^n,\alpha_{(2)}]} & & \Z_{(2)}^{t_{n,2l}'}
    \end{tikzcd}
    \]
    as we see e.g. from \cite[Lemma 4.1, A.1]{kaluzny2026highersmoothsurgerystructure}. Now, $[\Sigma^{2l}\cp^n,\alpha_{(2)}]$ is a split injection; see \cite[Section 5.C]{Madsen}. This implies that it is an isomorphism on the free parts. Thus: 
    \[
        \im[\Sigma^{2l}\cp^n,\Psi^3-\id]_{(2)}=\im([\Sigma^{2l}\cp^n,r_2]\circ[\Sigma^{2l}\cp^n,\alpha_{(2)}])=\im[\Sigma^{2l}\cp^n,r_2],
    \]
    proving the first statement. Furthermore, $\ker[\Sigma^{2l}\cp^n,r_2]$ is finite by \cite[Lemma 4.1]{kaluzny2026highersmoothsurgerystructure}, proving the second statement.

    In cases where $(l,n)\equiv(0,1),(4,3)$ $(\mod8,\mod4)$, the diagram is of the following form: 
    \begin{equation}
    \label{eq:diagram-splitting-with-Z2}
        \begin{tikzcd}
            & \Z_{(2)}^{t_{n,2l}'}\oplus2\text{-torsion} \arrow{dr}{[\Sigma^{2l}\cp^n,r_2]} & \\
            \Z_{(2)}^{t_{n,2l}'}\oplus\Z_2 \arrow{rr}{[\Sigma^{2l}\cp^n,\Psi^3-\id]_{(2)}} \arrow{ur}{[\Sigma^{2l}\cp^n,\alpha_{(2)}]} & & \Z_{(2)}^{t_{n,2l}'}\oplus\Z_2
        \end{tikzcd}
    \end{equation}
    Again, from \cite[Lemma 4.1]{kaluzny2026highersmoothsurgerystructure} it follows that $\im[\Sigma^{2l}\cp^n,r_2]\cong\Z_{(2)}^{t_{n,2l}'}$. The first statement now follows in the same way as above. Since $\ker[\Sigma^{2l}\cp^n,r_{2}]$ is finite, the second statement follows as well. 
\end{proof}
Now, we define the splitting $p_{(2)}$. In cases where $[\Sigma^{2l}\cp^n,\Psi^3-\id]_{(2)}$ is an isomorphism onto $\im[\Sigma^{2l}\cp^n,r_2]=\ker[\Sigma^{2l}\cp^n,J]_{(2)}$, we define 
\[
    p_{(2)}=[\Sigma^{2l}\cp^n,\alpha_{(2)}]\circ[\Sigma^{2l}\cp^n,\Psi^3-\id]_{(2)}^{-1}.
\]
In the remaining cases, we first factor the leftmost group by $\ker[\Sigma^{2l}\cp^n,\Psi^3-\id]_{(2)}$. Since $\im[\Sigma^{2l}\cp^n,r_2]\cong\Z_{(2)}^{t_{n,2l}'}$ by \cite[Lemma 4.1]{kaluzny2026highersmoothsurgerystructure}, we obtain a diagram
\[
    \begin{tikzcd}
        & \Z_{(2)}^{t_{n,2l}'}\oplus2\text{-torsion} \arrow{dr}{[\Sigma^{2l}\cp^n,r_2]} & \\
        \Z_{(2)}^{t_{n,2l}'} \arrow{rr}{[\Sigma^{2l}\cp^n,\Psi^3-\id]_{(2)}} \arrow{ur}{[\Sigma^{2l}\cp^n,\alpha_{(2)}]} & & \Z_{(2)}^{t_{n,2l}'}
    \end{tikzcd}
\]
where, by abuse of notation, $[\Sigma^{2l}\cp^n,\alpha_{(2)}]$ and $[\Sigma^{2l}\cp^n,\Psi^3-\id]_{(2)}$ are obtained from corresponding maps in \eqref{eq:diagram-splitting-with-Z2}, and $[\Sigma^{2l}\cp^n,\Psi^3-\id]_{(2)}$ is an isomorphism onto $\ker[\Sigma^{2l}\cp^n,J]_{(2)}$. 

We get: 
\begin{align*}
    \widetilde{H}^*(p_{(2)}[f_{free}])(\widetilde{K})&=\widetilde{H}^*(([\Sigma^{2l}\cp^n,\alpha_{(2)}]\circ[\Sigma^{2l}\cp^n,\Psi^3-\id]_{(2)}^{-1})[f_{free}])(\widetilde{K}) \\
    &=\widetilde{H}^*([\Sigma^{2l}\cp^n,\Psi^3-\id]_{(2)}^{-1}[ f_{free}])\widetilde{H}^*([\Sigma^{2l}\cp^n,\alpha_{(2)}])(\widetilde{K})
\end{align*}
By (\ref{Kvalues}), the discussion above it and Lemma \ref{KSWclass} we get for such $f$ that $[\Sigma^{2l}\cp^n,\Psi^3-\id]_{(2)}^{-1}[f_{free}]$ is in $[\Sigma^{2l}\cp^n,BO]$ (non-localized) the following:
\begin{align*}
    \sigma_{n,2l}(f)&=\textstyle\langle V_{\cp^n\times D^{2l}}^2\cup\sum_{r>1}P_{2^r-2}(w_i([\Sigma^{2l}\cp^n,\Psi^3-\id]_{(2)}^{-1}[f_{free}]),[\cp^n\times D^{2l},\cp^n\times S^{2l-1}]_{\Z_2}\rangle \\
    &+ \textstyle\langle V_{\cp^n\times D^{2l}}^2\cup\widetilde{H}^*(f_{tors})\widetilde{H}^*(\overline{i})(\sum_{r>1}\widetilde{k}_{2^r-2}),[\cp^n\times D^{2l},\cp^n\times S^{2l-1}]_{\Z_2}\rangle \\
    &+ \textstyle\langle V_{\cp^n\times D^{2l}}^2\cup\widetilde{H}^*(f_{S^{2l}})(\sum_{r>1}\widetilde{k}_{2^r-2}),[\cp^n\times D^{2l},\cp^n\times S^{2l-1}]_{\Z_2}\rangle.
\end{align*}
The bundles $[\Sigma^{2l}\cp^n,\Psi^3-\id]_{(2)}^{-1}[f_{free}]$ are over a suspension space, so all Stiefel-Whitney numbers in degree $2^r-2$ are zero except possibly $w_{2^r-2}$ (cup products are zero in cohomology). In addition, $\widetilde{H}^*(S^{2l})$ is non-zero only in degree $2l$, so
\begin{align}
\label{eq:final-equation-Arf-invariant}
    \sigma_{n,2l}(f)&=\textstyle\langle V_{\cp^n\times D^{2l}}^2\cup\sum_{r>1}w_{2^r-2}([\Sigma^{2l}\cp^n,\Psi^3-\id]_{(2)}^{-1}[f_{free}]),[\cp^n\times D^{2l},\cp^n\times S^{2l-1}]_{\Z_2}\rangle \\ \nonumber
    &+ \textstyle\langle V_{\cp^n\times D^{2l}}^2\cup\widetilde{H}^*(f_{tors})\widetilde{H}^*(\overline{i})(\sum_{r>1}\widetilde{k}_{2^r-2}),[\cp^n\times D^{2l},\cp^n\times S^{2l-1}]_{\Z_2}\rangle \\ \nonumber
    &+ \textstyle\langle V_{\cp^n\times D^{2l}}^2\cup\widetilde{H}^*(f_{S^{2l}})(\widetilde{k}_{2l}),[\cp^n\times D^{2l},\cp^n\times S^{2l-1}]_{\Z_2}\rangle
\end{align}

Let the isomorphism $\ker[\Sigma^{2l}\cp^n,J]\cong\Z^{t_{n,2l}'}$ be given by a choice of generators $\xi_1,\dots,\xi_{t_{n,2l}'}$ as in \cite[Remark 4.13, Table 3]{kaluzny2026highersmoothsurgerystructure}. The Adams operations, the preimage $[\Sigma^{2l}\cp^n,\Psi^3-\id]_{(2)}^{-1}$ of these generators, and the Stiefel-Whitney classes of generators of $[\Sigma^{2l}\cp^n,BO]$ are summarized in Table~\ref{tab:Adams-id-operations},\ref{tab:inverse-Adams-id-generators},\ref{tab:S-W-classes-generators}. These tables are computed in the following way.  

In order to obtain Table~\ref{tab:Adams-id-operations}, we need to compute the $3$rd Adams operation on the generators of $[\Sigma^{2l}\cp^n,BO]$; see \cite[Lemma A.1]{kaluzny2026highersmoothsurgerystructure}. We will present the computation for the generator $\mu_1\in[\Sigma^2\cp^4,BO]\cong\Z^2$. First, we compute the Pontryagin character of $\Psi^3\mu_1$ using \cite[Theorem 5.1,(vi)]{Adams_vect} and \cite[Lemma 4.8]{kaluzny2026highersmoothsurgerystructure}, and obtain
\[
    ph_q(\Psi^3\mu_1)=ch_{2q}(\Psi^3c(\mu_1))=3^{2q}ch_{2q}(c(\mu_1))=3^{2q}ph_q(\mu_1)
\]
We express this Pontryagin character as a linear combination of Pontryagin characters of the generators of $[\Sigma^{2l}\cp^n,BO]$. Since the Pontryagin character is a monomorphism in all cases in question, except when $l=2$ and $n=3$, this gives the precise linear combination of generators expressing $\Psi^3\mu_1$. When $l=2$ and $n=3$, the result follows from the case where $l=2$ and $n=5$ by naturality.  
Now, recall from \cite[Table 3]{kaluzny2026highersmoothsurgerystructure} that $\xi_1,\dots,\xi_{t_{n,2l}'}$ are linear combinations of generators of $[\Sigma^{2l}\cp^n,BO]$, which can be considered as elements in $[\Sigma^{2l}\cp^n,BO]_{(2)}$. We find the inverses of these elements under the matrices in Table~\ref{tab:Adams-id-operations} and obtain Table~\ref{tab:inverse-Adams-id-generators}. Finally, Table~\ref{tab:S-W-classes-generators} is obtained in the following way. The generators $\mu_i\mu_0^j$ are in the image of the realification map 
\[
    R:\widetilde{K}^0(\Sigma^{2l}\cp^n)\to \widetilde{KO}^0(\Sigma^{2l}\cp^n),
\]
their preimage can be determined using the Pontryagin character. We know that for (virtual) bundles in $\im R$, the total Stiefel-Whitney class is a$\mod2$ reduction of the total Chern class; see \cite[Problem~14-B]{milnor1974characteristic}. Thus, the total Stiefel-Whitney class can be computed from the Chern character of the generator's preimage under $R$. For instance, by \cite[Thm.2]{Fujii} we have $\mu_1=R(g\cdot\mu)$, and 
\[
    ch(g\cdot\mu)=ch(g\cdot(H-1))=y\times(e^x-1)
\]
so the total Chern class of $\mu_1\in[\Sigma^2\cp^6,BO]$ is
\[
    C(\mu_1)=1+(-x+x^2-x^3+x^4-x^5+x^6)\times y.
\]
The total Stiefel-Whitney class is the$\mod2$ reduction of this class.

Now, we compute the first summand of~\eqref{eq:final-equation-Arf-invariant}.
\begin{lema}
\label{obstrfreeodd}
    Let $[f_{free}]$ be given in coordinates by $x_1\xi_1+\dots+x_{t_{n,2l}'}\xi_{t_{n,2l}'}$. We have for $1\leq n\leq6$, $2\leq2l\leq6$:
    \begin{align*}
        \sigma_{n,2l}(p_{(2)}[f_{free}])=
        \begin{cases}
            x_1 &\text{for }l=1\\
            0 &\text{otherwise}
        \end{cases}
    \end{align*}
\end{lema}
\begin{proof}
    We describe the proof for the case where $l=1$ and $n=4$, the other cases follow analogously. From Table \eqref{tab:inverse-Adams-id-generators} we see that 
    \[
        [\Sigma^2\cp^4,\Psi^3-\id]_{(2)}^{-1}[f_{free}]=x_1(3\mu_1+2\mu_1\mu_0)+x_2(3\mu_1\mu_0).
    \]
    The Stiefel-Whitney classes are additive, since we are over a suspension space. We obtain:
    \[
        W([\Sigma^2\cp^4,\Psi^3-\id]_{(2)}^{-1}[f_{free}])=x_1(1+(x+x^2+x^3+x^4)\times y),
    \]
    where $x\in H^2(\cp^n)$ is the generator. Finally, from \eqref{eq:final-equation-Arf-invariant} and Table \ref{tab:Wu-classes} we get
    \begin{align*}
        \sigma_{4,2}(p_{(2)}[f_{free}])&=\textstyle\langle V_{\cp^4\times D^{2}}^2\cup\sum_{r>1}w_{2^r-2}([\Sigma^{2}\cp^4,\Psi^3-\id]_{(2)}^{-1}[f_{free}]),[\cp^4\times D^{2},\cp^4\times S^{1}]_{\Z_2}\rangle \\
        &=\textstyle\langle(1+x^2+x^4)\times1\cup x_1(x^2\times y),[\cp^4\times D^{2},\cp^4\times S^{1}]_{\Z_2}\rangle \\
        &=\textstyle x_1\langle x^4\times y,[\cp^4\times D^{2},\cp^4\times S^{1}]_{\Z_2}\rangle = x_1
    \end{align*}
\end{proof}
\begin{com}
    Atiyah and Hirzebruch have shown in Theorem 2 of \cite{Atiyah_Hirzebruch} that any vector bundle over the $9$-fold reduced suspension of any finite $CW$-complex has trivial Stiefel-Whitney classes. It follows that the surgery obstruction is zero on $p_{(2)}(\ker[\Sigma^{2l}\cp^n,J])\subseteq[\Sigma^{2l}\cp^n,G/O_{(2)}]$ for all $l\geq5$, $n+l$ odd.
\end{com}

The second summand of \eqref{eq:final-equation-Arf-invariant} is computed using the fact that the elements $f_{tors}$ ultimately come from $\pi_*^s$ via the spectral sequence in Subsection~\ref{subsect:results-spectral-sequence}. Since $\sigma_{n,2l}$ is a homomorphism with target $\Z_2$, it is zero on 
\[
    [\Sigma^{2l}\cp^n,\overline{i}]({}_p\coker[\Sigma^{2l}\cp^n,\Omega J])
\]
for odd $p$. On the $2$-primary part, we get the following results.
\begin{lema}
\label{obstrtors}
    The values of $\sigma_{n,2l}(\overline{i}\circ x_{n,l})$ for $1\leq n\leq6$ and $0\leq l\leq3$, where $[x_{n,l}]$ ranges over the generators of ${}_2\coker[\Sigma^{2l}\cp^n,\Omega J]$ are summarized in the following table: 
    \begin{table}[ht!]
        \centering
        \begin{tabular}{|c|c|c|c|c|c|c|} \hline
            \diaghead{\theadfont aaaaaaaaaa}%
            {$l$}{$n$} & $1$ & $3$ & $5$ \\ 
            \hline
            $0$ & $\eta^2\mapsto1$ & $\nu^2\mapsto1$ & $(\underline{\overline{\nu}},\mu_{10})\mapsto(0,0)$ \\
            \hline
            $2$ & $\nu^2\mapsto1$ & $(\underline{\nu^2}_2,\mu_{10})\mapsto(0,0)$ & $(\underline{\nu^2}_4,\sigma^2,\kappa)\mapsto(1,1,0)$ \\
            \hline
            \diaghead{\theadfont aaaaaaaaaa}%
            {$l$}{$n$} & $2$ & $4$ & $6$ \\ 
            \hline
            $1$ & $\nu^2\mapsto1$ & $\underline{\nu^2}_2\mapsto1$ & $(\underline{\nu^2}_4,\sigma^2,\kappa)\mapsto(1,1,0)$ \\
            \hline
            $3$ & $\cdot$ & $(\sigma^2,\kappa)\mapsto(1,0)$ & $(\nu_4,\underline{\sigma^2}_2,\underline{\kappa}_2)\mapsto(0,1,0)$ \\
            \hline
        \end{tabular}
        \label{tabobstrtors}
    \end{table}
\end{lema}
\begin{proof}
    We prove the case where $l=2$ and $n=5$, other cases follow analogously. From Lemma \ref{lema:coker-J-CP-values} we know that ${}_2\coker[\Sigma^4\cp^5,\Omega J]\cong\Z_4\oplus\Z_2^2$. The $\Z_2$-summands are generated by $[\Sigma^4j_{4},\coker J_2](\sigma^2)$ and $[\Sigma^4j_{4},\coker J_2](\kappa)$, and the $\Z_4$-summand is generated by $\underline{\nu^2}_4$. In the following, we will not differentiate between a class in $\pi_*(\coker J_2)$ and its representative map. 

    From \eqref{eq:final-equation-Arf-invariant} we have: 
    \begin{align*}
        &\textstyle \sigma_{5,4}(\overline{i}\circ(\sigma^2\circ\Sigma^4j_{4})) \\ 
        &\textstyle=\langle V_{\cp^5\times D^4}^2\cup\widetilde{H}^*(\sigma^2\circ\Sigma^4j_{4})\widetilde{H}^*(\overline{i})(\sum_{r>1}\Tilde{k}_{2^r-2}),[\cp^5\times D^4,\cp^5\times S^3]_{\Z_2}\rangle \\
        &\textstyle=\langle V_{\cp^5\times D^4}^2\cup\widetilde{H}^*(\Sigma^4j_{4})\widetilde{H}^*(\sigma^2)\widetilde{H}^*(\overline{i})(\sum_{r>1}\Tilde{k}_{2^r-2}),[\cp^5\times D^4,\cp^5\times S^3]_{\Z_2}\rangle
    \end{align*}
    Since the Kervaire invariant of $\sigma^2$ is non-zero by \cite[Ch. 8.1]{kochman2006stable} and the (reduced) $\Z_2$-comohology of $S^{14}$ is concentrated in degree $14$, $\widetilde{H}^*(\sigma^2)\widetilde{H}^*(\overline{i})(\sum_{r>1}\Tilde{k}_{2^r-2})$ is the unique non-zero class in $\widetilde{H}^{14}(S^{14};\Z_2)$. Furthermore, $\widetilde{H}^*(\Sigma^4j_{4})$ is an isomorphism in degree $14$, and it is the zero map in other degrees. Finally, from Table~\ref{tab:Wu-classes} we know that $V_{\cp^5\times D^4}^2=(1+x^4)\times1$. We obtain:
    \begin{align*}
        \sigma_{5,4}(\overline{i}\circ(\sigma^2\circ\Sigma^4j_4))&=\textstyle\langle (1+x^4)\times1\cup x^5\times y^{2},[\cp^5\times D^4,\cp^5\times S^3]_{\Z_2}\rangle \\
        &=\textstyle\langle x^5\times y^2,[\cp^5\times D^4,\cp^5\times S^3]_{\Z_2}\rangle=1
    \end{align*}
    The same argument shows that $\sigma_{5,4}(\overline{i}\circ(\kappa\circ\Sigma^4j_4))=0$, since the Kervaire invariant of $\kappa$ is zero \cite[Table 2]{Behrens-Hill-Hopkins-Mahowald}.

    For the last generator $\underline{\nu^2}_4$ we have again from \eqref{eq:final-equation-Arf-invariant}: 
    \begin{align*}
        \textstyle \sigma_{5,4}(\overline{i}\circ \underline{\nu^2}_4)=\langle V_{\cp^5\times D^4}^2\cup\widetilde{H}^*(\overline{i}\circ \underline{\nu^2}_4)(\sum_{r>1}\Tilde{k}_{2^r-2}),[\cp^5\times D^4,\cp^5\times S^3]_{\Z_2}\rangle
    \end{align*}
    Denote by $i_{m,n}:\cp^m\hookrightarrow\cp^n$ the standard inclusion, $[\Sigma^4i_{1,5},\coker J_2](\underline{\nu^2}_4)$ is then by definition equal to $[\Sigma^4j_0,\coker J_2](\nu^2)$. Thus, 
    \begin{align*}
        \widetilde{H}^*([\Sigma^4i_{1,5},\coker J_2](\underline{\nu^2}_4))\widetilde{H}^*(\overline{i})(\sum_{r>1}\Tilde{k}_{2^r-2})&=\widetilde{H}^*(\nu^2\circ\Sigma^4j_0)\widetilde{H}^*(\overline{i})(\sum_{r>1}\Tilde{k}_{2^r-2}) \\
        &= x\times y^2
    \end{align*}
    since the Kervaire invariant of $\nu^2$ is non-zero by \cite[Ch. 8.1]{kochman2006stable}. Now, $\widetilde{H}^*(\Sigma^4i_{1,5})$ is an isomorphism up to degree $6$, so 
    \begin{align*}
        \widetilde{H}^*(\underline{\nu^2}_4)\widetilde{H}^*(\overline{i})(\sum_{r>1}\Tilde{k}_{2^r-2})&=x\times y^2 +\text{higher degree terms}.
    \end{align*}
    The only possible non-zero higher degree term is $\widetilde{H}^*(\underline{\nu^2}_4)\widetilde{H}^*(\overline{i})(\Tilde{k}_{14})$. However, $\underline{\nu^2}_4$ is by definition determined only up to $\ker[\Sigma^4i_{1,5},\coker J_2]$. Since the first generator $[\Sigma^4j_{4},\coker J_2](\sigma^2)$ belongs to $\ker[\Sigma^4i_{1,5},\coker J_2]$ and 
    \[
        \widetilde{H}^*(\Sigma^4j_{4})\widetilde{H}^*(\sigma^2)\widetilde{H}^*(\overline{i})\Tilde{k}_{14}\neq0
    \]
    we can choose the extension $\underline{\nu^2}_4$ so that $\widetilde{H}^*(\underline{\nu^2}_4)\widetilde{H}^*(\overline{i})\Tilde{k}_{14}=0$. We get
    \begin{align*}
        \textstyle \sigma_{5,4}(\overline{i}\circ \underline{\nu^2}_4)&=\textstyle\langle (1+x^4)\times1\cup x\times y^2,[\cp^5\times D^4,\cp^5\times S^3]_{\Z_2}\rangle \\
        &=\textstyle\langle x^5\times y^2,[\cp^5\times D^4,\cp^5\times S^3]_{\Z_2}\rangle=1 
    \end{align*}
\end{proof}
Finally, recall that $\pi_4(G/O)\cong\Z$, while $\pi_{2}(G/O)$ and $\pi_6(G/O)$ are isomorphic to $\Z_2$.
\begin{lema}
\label{lema:Arf-inv-on-sphere-part}
    Let $[f_{S^{2l}}]=y_{2l}\gamma_{2l}$ where $\gamma_{2l}$ denotes the generator of $\pi_{2l}(G/O)$ for $1\leq l\leq3$, $y_{2l}\in\Z_2$ for $l=1,3$ and $y_{2l}\in\Z$ for $l=2$. We have for $1\leq n\leq6$
    \[
        \sigma_{n,2l}([f_{S^{2l}}])=
        \begin{cases}
            y_{2l} & \text{if } l=1\text{ or }3 \\
            0 & \text{if } l=2
        \end{cases}
    \]
\end{lema}
\begin{proof}
    We prove the cases where $(l,n)=(1,6)$ and $(2,5)$, the remaining ones follow analogously. We start with the first case. Here $\pi_2(G/O)\cong\pi_2(G)\cong\pi_2^s\cong\Z_2$ with generator $\gamma_2=\pi_2(\overline{i})(\eta^2)=\overline{i}\circ\eta^2$. From \eqref{eq:final-equation-Arf-invariant} we get
    \begin{align*}
        \sigma_{6,2}([f_{S^{2}}])&=\textstyle\langle V_{\cp^6\times D^{2}}^2\cup\widetilde{H}^*(f_{S^{2}})(\widetilde{k}_{2}),[\cp^6\times D^{2},\cp^6\times S^{1}]_{\Z_2}\rangle \\
        &=\textstyle\langle(1+x^2+x^6)\times1\cup y_2\widetilde{H}^*(\overline{i}\circ\eta^2)(\widetilde{k}_{2}),[\cp^6\times D^{2},\cp^6\times S^{1}]_{\Z_2}\rangle, 
    \end{align*}
    and since the Kervaire invariant of $\eta^2$ is non-zero by \cite[Ch. 8.1]{kochman2006stable}, we obtain
    \begin{align*}
        \sigma_{6,2}([f_{S^{2}}])&= \textstyle\langle(1+x^2+x^6)\times1\cup y_2(1\times y),[\cp^6\times D^{2},\cp^6\times S^{1}]_{\Z_2}\rangle \\
        &=\textstyle y_2\langle x^6\times y,[\cp^6\times D^{2},\cp^6\times S^{1}]_{\Z_2}\rangle=y_2.
    \end{align*}

    In the second case, we have $\pi_4(G/O)\cong\ker\pi_4(J)\cong\Z$ with generator $\gamma_4$, so $[f_{S^{4}}]=y_4\gamma_4$ for $y_4\in\Z$. The cohomology class $\widetilde{H}^*(f_{S^{2l}})(\widetilde{k}_{2l})$ is some multiple $c(1\times y^{2})$ where $c\in\Z$. From \eqref{eq:final-equation-Arf-invariant} we obtain:
    \begin{align*}
        \sigma_{5,4}([f_{S^{4}}])&= \textstyle\langle(1+x^4)\times1\cup y_4c\cdot(1\times y^2),[\cp^5\times D^{4},\cp^5\times S^{3}]_{\Z_2}\rangle=0.
    \end{align*}
\end{proof}
\begin{com}
\label{rem:diagrams-of-arf-invariants}
The results of the last three lemmas are summarized in the following diagrams involving generators of summands of $\sN_{\partial}^{\DIFF}(\cpdt)$, where a full node represents a non-zero surgery obstruction, an empty node suggests that $\sigma_{n,2l}^{\DIFF}=0$, and an edge signals that the bottom generator maps to the other under the map induced by the standard inclusion
\[
    \sN_{\partial}^{\DIFF}(i_{n-2,n}):\sN_{\partial}^{\DIFF}(\cpdt)\to\sN_{\partial}^{\DIFF}(\cp^{n-2}\times\text{D}^{2l}).
\]
Note that we make the specific choice of generators $\xi_i$ as described in \cite[Remark 4.13]{kaluzny2026highersmoothsurgerystructure}. 
In each diagram we denote (from left to right) first the generators of the free part (generators of $\ker[\Sigma^{2l}\cp^n,J]$, unless $l=2$ when they are preceded by the generator of $\pi_{4}(G/O)\cong\Z$), then the generator of $\pi_{2l}(G/O)$ for $l=1,3$, thirdly the generators of ${}_2\coker[\Sigma^{2l}\cp^n,\Omega J]$ and lastly those of ${}_3\coker[\Sigma^{2l}\cp^n,\Omega J]$.
\newline\noindent
\begin{tikzpicture}
    \node[minimum size=1pt] (L1) at (0,2) {$l=1$:};
    \node[minimum size=1pt] (N5) at (0.5,-0.5) {$n$};
    \node[minimum size=1pt] (N1) at (0.5,0) {$6$};
    \node[minimum size=1pt] (N2) at (0.5,0.5) {$4$};
    \node[minimum size=1pt] (N3) at (0.5,1) {$2$};
    \node[minimum size=1pt] (N4) at (0.5,1.5) {$0$};
    \node[minimum size=1pt] (X1) at (1,-0.5) {$\xi_1$};
    \node[shape=circle,fill=black,minimum size=1pt] (A1) at (1,0) {}; 
    \node[shape=circle,fill=black,minimum size=1pt] (A2) at (1,0.5) {};
    \node[shape=circle,fill=black,minimum size=1pt] (A3) at (1,1) {};
    \node[shape=circle,draw=black,minimum size=1pt] (B1) at (1.5,0) {};
    \node[shape=circle,draw=black,minimum size=1pt] (B2) at (1.5,0.5) {};
    \node[minimum size=1pt] (X2) at (1.5,-0.5) {$\xi_2$};
    \node[shape=circle,draw=black,minimum size=1pt] (C1) at (2,0) {};
    \node[minimum size=1pt] (X3) at (2,-0.5) {$\xi_3$};
    \draw (A1) -- (A2) -- (A3); \draw (B1) -- (B2) ;
    \node[shape=circle,fill=black,minimum size=1pt] (D1) at (3,0) {};
    \node[shape=circle,fill=black,minimum size=1pt] (D2) at (3,0.5) {};
    \node[shape=circle,fill=black,minimum size=1pt] (D3) at (3,1) {};
    \node[shape=circle,fill=black,minimum size=1pt] (D4) at (3,1.5) {};
    \node[minimum size=1pt] (Y1) at (3,-0.5) {$\eta^2$};
    \draw (D1) -- (D2) -- (D3) -- (D4);
    \node[shape=circle,draw=black,minimum size=1pt] (G1) at (4,0) {};
    \node[shape=circle,fill=black,minimum size=1pt] (G2) at (4,0.5) {};
    \node[shape=circle,fill=black,minimum size=1pt] (G3) at (4,1) {};
    \node[minimum size=1pt] (Z3) at (4,-0.5) {$\underline{\nu^2}_4$};
    \node[minimum size=1pt] (Z4) at (4,1.5) {$\nu^2$};
    \draw (G1) -- (G2) -- (G3);
    \node[shape=circle,fill=black,minimum size=1pt] (E1) at (4.5,0) {};
    \node[minimum size=1pt] (Z1) at (4.5,-0.5) {$\sigma^2$};
    \node[shape=circle,draw=black,minimum size=1pt] (F1) at (5,0) {};
    \node[minimum size=1pt] (Z2) at (5,-0.5) {$\kappa$};
    \node[shape=circle,draw=black,minimum size=1pt] (H1) at (6,0.5) {};
    \node[minimum size=1pt] (W1) at (6,0) {$\beta_1$};

    \node[minimum size=1pt] (L11) at (6.5,2) {$l=2$:};
    \node[minimum size=1pt] (N55) at (7,-0.5) {$n$};
    \node[minimum size=1pt] (N11) at (7,0) {$5$};
    \node[minimum size=1pt] (N22) at (7,0.5) {$3$};
    \node[minimum size=1pt] (N33) at (7,1) {$1$};
    \node[shape=circle,draw=black,minimum size=1pt] (D11) at (7.5,0) {};
    \node[shape=circle,draw=black,minimum size=1pt] (D22) at (7.5,0.5) {};
    \node[shape=circle,draw=black,minimum size=1pt] (D33) at (7.5,1) {};
    \node[minimum size=1pt] (Y11) at (7.5,-0.5) {$\gamma_4$};
    \draw (D11) -- (D22) -- (D33);
    \node[minimum size=1pt] (X11) at (8.5,-0.5) {$\xi_1$};
    \node[shape=circle,draw=black,minimum size=1pt] (A11) at (8.5,0) {}; 
    \node[shape=circle,draw=black,minimum size=1pt] (A22) at (8.5,0.5) {};
    \node[shape=circle,draw=black,minimum size=1pt] (B11) at (9,0) {};
    \node[minimum size=1pt] (X22) at (9,-0.5) {$\xi_2$};
    \draw (A11) -- (A22);
    \node[shape=circle,fill=black,minimum size=1pt] (E11) at (10,0) {};
    \node[shape=circle,draw=black,minimum size=1pt] (E22) at (10,0.5) {};
    \node[shape=circle,fill=black,minimum size=1pt] (E33) at (10,1) {};
    \node[minimum size=1pt] (Z11) at (10,-0.5) {$\underline{\nu^2}_4$};
    \node[minimum size=1pt] (Z22) at (10,1.5) {$\nu^2$};
    \draw (E11) -- (E22) -- (E33);
    \node[shape=circle,fill=black,minimum size=1pt] (F11) at (10.5,0) {};
    \node[minimum size=1pt] (Z33) at (10.5,-0.5) {$\sigma^2$};
    \node[shape=circle,draw=black,minimum size=1pt] (F22) at (10.5,0.5) {};
    \node[minimum size=1pt] (Z33) at (10.5,1) {$\mu_{10}$};
    \draw (B11) -- (F22);
    \node[shape=circle,draw=black,minimum size=1pt] (G11) at (11,0) {};
    \node[minimum size=1pt] (Z44) at (11,-0.5) {$\kappa$};
    \node[shape=circle,draw=black,minimum size=1pt] (H11) at (12,0) {};
    \node[shape=circle,draw=black,minimum size=1pt] (H22) at (12,0.5) {};
    \node[minimum size=1pt] (W11) at (12,1) {$\beta_1$};
    \node[minimum size=1pt] (W22) at (12,-0.5)
    {$\underline{\beta_1}_2$};
    \draw (H11) -- (H22);
\end{tikzpicture}
\\
\begin{tikzpicture}
    \node[minimum size=1pt] (L1) at (0,2) {$l=3$:};
    \node[minimum size=1pt] (N5) at (0.5,-0.5) {$n$};
    \node[minimum size=1pt] (N1) at (0.5,0) {$6$};
    \node[minimum size=1pt] (N2) at (0.5,0.5) {$4$};
    \node[minimum size=1pt] (N3) at (0.5,1) {$2$};
    \node[minimum size=1pt] (N4) at (0.5,1.5) {$0$};
    \node[minimum size=1pt] (X1) at (1,-0.5) {$\xi_1$};
    \node[shape=circle,draw=black,minimum size=1pt] (A1) at (1,0) {}; 
    \node[shape=circle,draw=black,minimum size=1pt] (A2) at (1,0.5) {};
    \node[shape=circle,draw=black,minimum size=1pt] (A3) at (1,1) {};
    \node[shape=circle,draw=black,minimum size=1pt] (B1) at (1.5,0) {};
    \node[shape=circle,draw=black,minimum size=1pt] (B2) at (1.5,0.5) {};
    \node[minimum size=1pt] (X2) at (1.5,-0.5) {$\xi_2$};
    \node[shape=circle,draw=black,minimum size=1pt] (C1) at (2,0) {};
    \node[minimum size=1pt] (X3) at (2,-0.5) {$\xi_3$};
    \draw (A1) -- (A2) -- (A3); \draw (B1) -- (B2) ;
    \node[shape=circle,fill=black,minimum size=1pt] (D1) at (3,0) {};
    \node[shape=circle,fill=black,minimum size=1pt] (D2) at (3,0.5) {};
    \node[shape=circle,fill=black,minimum size=1pt] (D3) at (3,1) {};
    \node[shape=circle,fill=black,minimum size=1pt] (D4) at (3,1.5) {};
    \node[minimum size=1pt] (Y1) at (3,-0.5) {$\nu^2$};
    \draw (D1) -- (D2) -- (D3) -- (D4);
    \node[shape=circle,fill=black,minimum size=1pt] (E1) at (4,0) {};
    \node[shape=circle,fill=black,minimum size=1pt] (E2) at (4,0.5) {};
    \node[minimum size=1pt] (Z1) at (4,1) {$\sigma^2$};
    \node[minimum size=1pt] (Z11) at (4,-0.5) {$\underline{\sigma^2}_2$};
    \draw (E1) -- (E2);
    \node[shape=circle,draw=black,minimum size=1pt] (F1) at (4.5,0) {};
    \node[shape=circle,draw=black,minimum size=1pt] (F2) at (4.5,0.5) {};
    \node[minimum size=1pt] (Z2) at (4.5,1) {$\kappa$};
    \node[minimum size=1pt] (Z22) at (4.5,-0.5) {$\underline{\kappa}_2$};
    \draw (F1) -- (F2);
    \node[shape=circle,draw=black,minimum size=1pt] (G1) at (5,0) {};
    \node[minimum size=1pt] (Z3) at (5,-0.5) {$\nu_4$};
    \node[shape=circle,draw=black,minimum size=1pt] (H1) at (6,1) {};
    \node[minimum size=1pt] (W1) at (6,0.5) {$\beta_1$};
\end{tikzpicture} 

Note that the only case where the free part interacts with the torsion is for $l=2$, where the generator of a $\Z$-summand $\xi_2$ maps to the Adams element of order two $\mu_{10}$ under the map
\[
    [\Sigma^4i_{3,4},G/O]:[\Sigma^4\cp^4,G/O]\to[\Sigma^4\cp^3,G/O].
\]
This follows easily from the exact sequence
\[
 \dots\to\pi_{12}(G/O)\to[\Sigma^4\cp^4,G/O]\to[\Sigma^4\cp^3,G/O]\to\pi_{11}(G/O)\to\dots
\]
\end{com}

%% file: proofs.tex
\section{Proofs of Main Theorems}
\label{sec:proofs}

\begin{proof}[Proof of Theorem~\ref{thm:main-theorem-1}]
    From the surgery exact sequence of $\cp^n$ we see that if $k$ is even, $\sS_{\partial}^{\DIFF}(\cpd)\cong\ker\sigma_{n,k}^{\DIFF}$, and for odd $k$, $\sS_{\partial}^{\DIFF}(\cpd)$ is an extension of $\sN_{\partial}^{\DIFF}(\cpd)$ by $\coker\sigma_{n,k+1}^{\DIFF}$; see \cite[Sec.3]{kaluzny2026highersmoothsurgerystructure}.

    Note that for $2n+k\equiv0\mod4$, $\sigma_{n,k}^{\DIFF}$ is zero on torsion by \cite[Lemma~5.3]{kaluzny2026highersmoothsurgerystructure}. In cases where $k$ is even, the results now follow directly from Lemmas~\ref{lema:coker-J-CP-values}, \ref{obstrtors} and \ref{lema:Arf-inv-on-sphere-part} (see also the diagram in Remark \ref{rem:diagrams-of-arf-invariants}).

    In cases where $k$ is odd, $\sN_{\partial}^{\DIFF}(\cpd)$ is isomorphic to $\coker[\Sigma^k\cp^n,\Omega J]$ by \cite[Lemma 4.1, Cor. 4.2]{kaluzny2026highersmoothsurgerystructure}, its values are summarized in the table of Lemma~\ref{lema:coker-J-CP-values}. Finally, we observe from Lemmas~\ref{obstrfreeodd}, \ref{obstrtors} and \ref{lema:Arf-inv-on-sphere-part} that for $2n+k\equiv1\mod4$, $\coker\sigma_{n,k+1}^{\DIFF}$ is zero unless $k=n=3$ when it is isomorphic to $\Z_2$. The values of the surgery obstruction
    \[
        \sigma_{n,k+1}^{\DIFF}:\sN_{\partial}^{\DIFF}(\cp^n\times\text{D}^{k+1})\to L_{2n+k+1}(\Z)\cong\Z
    \]
    for $2n+k\equiv3\mod4$ are summarized in \cite[Tables 5,6,7]{kaluzny2026highersmoothsurgerystructure}, and in these cases $\coker\sigma_{n,k+1}^{\DIFF}\cong\Z_{d}$ where $d$ is the greatest common divisor of the coefficients in the formula for $\sigma_{n,k+1}^{\DIFF}$.
\end{proof}
\begin{proof}[Proof of Theorem~\ref{thm:main-theorem-2}]
    The entries in the matrix $F_{n,l}'$ are obtained by computing the splitting invariants of elements from $\sN_{\partial}^{\DIFF}(\cpdt)$; see \cite[Section~6]{kaluzny2026highersmoothsurgerystructure}. We prove the theorem for $l=1$ and $n=6$. We start with the matrix $B_{6,1}'$. Its $(i,j)$-th entry is given by the $\Z_2$-valued surgery obstruction of the image of $\xi_j$ under the map
    \[
        \sN_{\partial}^{\DIFF}(i_{i-1,n}):\sN_{\partial}^{\DIFF}(\cpdt)\to\sN_{\partial}^{\DIFF}(\cp^{i-1}\times\text{D}^{2l}).
    \]
    The matrix can thus be reconstructed from the diagram in Remark~\ref{rem:diagrams-of-arf-invariants}, where each $\xi_i$ spans a column of the matrix. A full node represents an entry of $1$ in the column, whereas an empty or missing node represents a zero. We obtain: 
    \[
    B_{6,1}'=
    \begin{pmatrix}
        0 & 0 & 0 \\
        1 & 0 & 0 \\
        1 & 0 & 0 \\
        1 & 0 & 0
    \end{pmatrix}.
    \]
        
    The entries of the matrix $C_{6,1}'$ are given by the $\Z_2$-valued surgery obstructions of the images of generators of the torsion part 
    \[
        \coker[\Sigma^2\cp^6,\Omega J]\oplus\pi_2(G/O)\cong\Z_2^3\{\ext^4(\nu^2),\sigma^2,\kappa\}\oplus\Z_2\{\eta^2\}.
    \] 
    Fix the basis of this torsion part to be $\{\eta^2,\ext^4(\nu^2),\sigma^2,\kappa\}$. Again, each of these elements spans a column of the matrix, and the entries are obtained from the diagram in Remark~\ref{rem:diagrams-of-arf-invariants} as above. We obtain:
    \[
    C_{6,1}'=
    \begin{pmatrix}
        1 & 0 & 0 & 0 \\
        1 & 1 & 0 & 0 \\
        1 & 1 & 0 & 0 \\
        1 & 0 & 1 & 0
    \end{pmatrix}.
    \]
\end{proof}
\begin{proof}[Proof of Theorem~\ref{thm:main-theorem-3}, \ref{thm:main-theorem-5} and \ref{thm:main-theorem-4}]
    The matrix $E_{n,l}$ describes the inclusion of the subgroup $\ker\sigma_{n,2l}^{\DIFF}$ into $\sN_{\partial}^{\DIFF}(\cpdt)$. We prove the theorems in the case where $l=1$, $n=6$. 
    
    We start with $P_{6,1}$ and $Q_{6,1}$. These matrices describe the map $\Z^3\to\Z^3\oplus\Z_2^4$ given by the inclusion of the free part of $\ker\sigma_{6,2}^{\DIFF}$ into $\sN_{\partial}^{\DIFF}(\cp^6\times\text{D}^2)$. From Lemma~\ref{obstrfreeodd} and \ref{lema:Arf-inv-on-sphere-part} (or from the diagram in Remark~\ref{rem:diagrams-of-arf-invariants}) we see that $\xi_1+\eta^2,\xi_2,\xi_3$ generate the free part of $\ker\sigma_{6,2}^{\DIFF}$, thus 
    \[
    P_{6,1}=
    \begin{pmatrix}
        1 & 0 & 0 \\
        0 & 1 & 0 \\
        0 & 0 & 1
    \end{pmatrix}\quad\text{and}\quad Q_{6,1}=
    \begin{pmatrix}
        1 & 0 & 0 \\
        0 & 0 & 0 \\
        0 & 0 & 0 \\
        0 & 0 & 0
    \end{pmatrix}.
    \]
    Finally, the matrix $R_{6,1}$ describes the map $\Z_2^3\to\Z_2^4$ given by the inclusion of the torsion part of $\ker\sigma_{6,2}^{\DIFF}$ into the torsion part of $\sN_{\partial}^{\DIFF}(\cp^6\times\text{D}^2)$. From Lemma~\ref{obstrtors} and \ref{lema:Arf-inv-on-sphere-part} (or from the diagram in Remark~\ref{rem:diagrams-of-arf-invariants}) we see that the elements $\eta^2+\sigma^2,\ext^4(\nu^2),\kappa$ generate the torsion part of $\ker\sigma_{6,2}^{\DIFF}$. So, with respect to the basis $\{\eta^2,\ext^4(\nu^2),\sigma^2, \kappa\}$, we obtain:
    \[
        R_{6,1}=
        \begin{pmatrix}
            1 & 0 & 0 \\
            0 & 1 & 0 \\
            1 & 0 & 0 \\
            0 & 0 & 1
        \end{pmatrix}\qedhere
    \]
\end{proof}
\begin{proof}[Proof of Corollary~\ref{cor:splitting-invariants}]
    We need to prove that the congruences in Table~\ref{tab:splitting-invariants} characterize the image of the forgetful map $F_{\cpdt}$ from Diagram~\eqref{eqn:diagram-forgetful-map}, resp. the image of the composition $\eta_{n,2l}^{\TOP}\circ F_{\cpdt}$. We prove this for $l=1$, $n=6$.

    From Diagram~\eqref{eqn:diagram-forgetful-map} and \eqref{eqn:matrix-forgetful-map}, we see that the composition $\eta_{6,2}^{\TOP}\circ F_{\cp^6\times\text{D}^2}$ is described by the matrix 
    \[
        F_{6,1}=F_{6,1}'\cdot E_{6,1}=
        \begin{pmatrix}
            A_{6,1}' & 0\\
            B_{6,1}' & C_{6,1}'
        \end{pmatrix}\cdot
        \begin{pmatrix}
            P_{6,1} & 0 \\
            Q_{6,1} & R_{6,1}
        \end{pmatrix}.
    \] 
    From \cite[Table 1]{kaluzny2026highersmoothsurgerystructure} and Theorem~\ref{thm:main-theorem-2}-\ref{thm:main-theorem-4} we obtain:
    \begin{align*}
        F_{6,1}&=
        \begin{pmatrix}
            -2 & 0 & 0 & 0 & 0 & 0 & 0 \\
            44 & 28 & 0 & 0 & 0 & 0 & 0 \\
            -662 & -632 & -992 & 0 & 0 & 0 & 0 \\
            0 & 0 & 0 & 1 & 0 & 0 & 0 \\
            1 & 0 & 0 & 1 & 1 & 0 & 0 \\
            1 & 0 & 0 & 1 & 1 & 0 & 0 \\
            1 & 0 & 0 & 1 & 0 & 1 & 0
        \end{pmatrix}\cdot
        \begin{pmatrix}
            1 & 0 & 0 & 0 & 0 & 0 \\
            0 & 1 & 0 & 0 & 0 & 0 \\
            0 & 0 & 1 & 0 & 0 & 0 \\
            1 & 0 & 0 & 1 & 0 & 0 \\
            0 & 0 & 0 & 0 & 1 & 0 \\
            0 & 0 & 0 & 1 & 0 & 0 \\
            0 & 0 & 0 & 0 & 0 & 1
        \end{pmatrix} \\
        &=
        \begin{pmatrix}
            -2 & 0 & 0 & 0 & 0 & 0 \\
            44 & 28 & 0 & 0 & 0 & 0 \\
            -662 & -632 & -992 & 0 & 0 & 0 \\
            1 & 0 & 0 & 1 & 0 & 0 \\
            0 & 0 & 0 & 1 & 1 & 0 \\
            0 & 0 & 0 & 1 & 1 & 0 \\
            0 & 0 & 0 & 0 & 0 & 0
        \end{pmatrix}.
    \end{align*}
    Here, the first three rows correspond to integral splitting invariants $\overline{\sigma}_{1,2},\overline{\sigma}_{3,2},\overline{\sigma}_{5,2}$ and the last four rows to $\Z_2$-valued splitting invariants $\overline{\sigma}_{2i,2}$, $i=0,...,3$. Denote by $(x_1,x_2,x_3,y_1,y_2,y_3)$ an element of $\sS_{\partial}^{\DIFF}(\cp^6\times\text{D}^2)\cong\Z^3\oplus\Z_2^3$. We obtain the following system of equations:
    \begin{align*}
        -2x_1&=\overline{\sigma}_{1,2} \\
        44x_1+28x_2&=\overline{\sigma}_{3,2} \\
        -662x_1-632x_2-992x_3&=\overline{\sigma}_{5,2} \\
        x_1+y_1&\equiv\overline{\sigma}_{0,2}\mod2 \\
        y_1+y_2&\equiv\overline{\sigma}_{2,2}\mod2 \\
        y_1+y_2&\equiv\overline{\sigma}_{4,2}\mod2 \\
        0&\equiv\overline{\sigma}_{6,2}\mod2
    \end{align*}
    Finally, we rearrange the system to express $x_i,y_i$ in terms of $\overline{\sigma}_{j,2}$ and obtain the congruences in Table~\ref{tab:splitting-invariants}. Note that we do not include the last equation in the table, since any element from $\sS_{\partial}^{\TOP}(\cp^6\times\text{D}^2)$ satisfies it automatically.
\end{proof}

%% file: appendix.tex
\section{}

\begin{table}[ht!]
    \centering
    \begin{tabular}{|c|c|c|c|c|c|c|}
        \hline
        $n$ & $2$ & $3$ & $4$ & $5$ & $6$ \\
        \hline
        $B_{n,1}'$ & 
        $
        \begin{pmatrix}
            0 \\
            1
        \end{pmatrix}
        $
        & 
        $
        \begin{pmatrix}
            0 & 0 \\
            1 & 0
        \end{pmatrix}
        $
        & 
        $
        \begin{pmatrix}
            0 & 0 \\
            1 & 0 \\
            1 & 0 
        \end{pmatrix}
        $
        &
        $
        \begin{pmatrix}
            0 & 0 & 0 \\
            1 & 0 & 0 \\
            1 & 0 & 0
        \end{pmatrix}
        $
        &
        $
        \begin{pmatrix}
            0 & 0 & 0 \\
            1 & 0 & 0 \\
            1 & 0 & 0 \\
            1 & 0 & 0
        \end{pmatrix}
        $ \\
        \hline
    \end{tabular}
    \caption{Table of matrices $B_{n,1}'$}
    \label{tab:matrix-B_n,1'}
\end{table}

\begin{table}[ht!]
    \centering
    \begin{tabular}{|c|c|c|c|c|c|c|} \hline            \diaghead{\theadfont aaaaaa}%
        {$n$}{$l$} & $1$ & $2$ & $3$ \\
        \hline
        $1$ & $\blank$ &
        $
        \begin{pmatrix}
            1
        \end{pmatrix}
        :\Z_2\to\Z_2
        $
        & 
        $
        \begin{pmatrix}
            1 & 0
        \end{pmatrix}
        :\Z_2^2\to\Z_2
        $
        \\
        \hline
        $2$ & 
        $
        \begin{pmatrix}
            1 & 0 \\
            1 & 1
        \end{pmatrix}
        :\Z_2^2\to\Z_2^2
        $
        &
        $
        \begin{pmatrix}
            1 
        \end{pmatrix}
        :\Z_4\to\Z_2
        $
        & 
        $
        \begin{pmatrix}
            1 & 0 \\
            1 & 0
        \end{pmatrix}
        :\Z_2\oplus\Z_3\to\Z_2^2
        $ 
        \\
        \hline
        $3$ &
        $
        \begin{pmatrix}
            1 & 0 & 0 \\
            1 & 1 & 0
        \end{pmatrix}
        :\Z_2^3\to\Z_2^2
        $
        &
        $
        \begin{pmatrix}
            1 & 0 & 0 \\
            0 & 0 & 0 
        \end{pmatrix}
        :\Z_4\oplus\Z_2\oplus\Z_3\to\Z_2^2
        $
        & 
        $
        \begin{pmatrix}
            1 & 0 \\
            1 & 0 
        \end{pmatrix}
        :\Z_2\oplus\Z_3\to\Z_2^2
        $
        \\
        \hline
        $4$ &
        $
        \begin{pmatrix}
            1 & 0 & 0 \\
            1 & 1 & 0 \\
            1 & 1 & 0 \\
        \end{pmatrix}
        :\Z_2^2\oplus\Z_3\to\Z_2^3
        $
        &
        $
        \begin{pmatrix}
            1 & 0 \\
            0 & 0
        \end{pmatrix}
        :\Z_4\oplus\Z_3\to\Z_2^2
        $
        & 
        $
        \begin{pmatrix}
            1 & 0 & 0 \\
            1 & 0 & 0 \\
            1 & 1 & 0
        \end{pmatrix}
        :\Z_2^3\to\Z_2^3
        $
        \\
        \hline
        $5$ &
        $
        \begin{pmatrix}
            1 & 0 & 0 \\
            1 & 1 & 0 \\
            1 & 1 & 0
        \end{pmatrix}
        :\Z_2^2\oplus\Z_3\to\Z_2^3
        $
        &
        $
        \begin{pmatrix}
            1 & 0 & 0 & 0 \\
            0 & 0 & 0 & 0 \\
            1 & 1 & 0 & 0 
        \end{pmatrix}
        :\Z_4\oplus\Z_2^2\oplus\Z_3\to\Z_2^3
        $
        & 
        $
        \begin{pmatrix}
            1 & 0 & 0 & 0 \\
            1 & 0 & 0 & 0 \\
            1 & 1 & 0 & 0
        \end{pmatrix}
        :\Z_2^4\to\Z_2^3
        $
        \\
        \hline
        $6$ & 
        $
        \begin{pmatrix}
            1 & 0 & 0 & 0 \\
            1 & 1 & 0 & 0 \\
            1 & 1 & 0 & 0 \\
            1 & 0 & 1 & 0 
        \end{pmatrix}
        :\Z_2^4\to\Z_2^4
        $
        & 
        $
        \begin{pmatrix}
            1 & 0 & 0 \\
            0 & 0 & 0 \\
            1 & 1 & 0
        \end{pmatrix}
        :\Z_4^2\oplus\Z_3\to\Z_2^3
        $
        & 
        $
        \begin{pmatrix}
            1 & 0 & 0 & 0 \\
            1 & 0 & 0 & 0 \\
            1 & 1 & 0 & 0 \\
            1 & 1 & 0 & 0
        \end{pmatrix}
        :\Z_2^3\oplus\Z_4\to\Z_2^4
        $
        \\
        \hline
    \end{tabular}
    \caption{Table of matrices $C_{n,l}'$} 
\label{tab:matrix-C_n,l'}
\end{table}

\begin{table}[ht!]
    \centering
    \begin{tabular}{|c|c|c|c|}
        \hline
        $n$ & $2$ & $4$ & $6$ \\
        \hline
        $Q_{n,1}$ & 
        $
        \begin{pmatrix}
            1 \\
            0
        \end{pmatrix}:\Z\to\Z_2^2
        $
        & 
        $
        \begin{pmatrix}
            1 & 0 \\
            0 & 0 \\
            0 & 0 
        \end{pmatrix}:\Z^2\to\Z_2^2\oplus\Z_3
        $
        &
        $
        \begin{pmatrix}
            1 & 0 & 0 \\
            0 & 0 & 0 \\
            0 & 0 & 0 \\
            0 & 0 & 0
        \end{pmatrix}:\Z^3\to\Z_2^4
        $ \\
        \hline
    \end{tabular}
    \caption{Table of matrices $Q_{n,1}$}
    \label{tab:matrix-Q_n,1}
\end{table}

\newpage

\begin{table}[ht!]
    \centering
    \begin{tabular}{|c|c|c|c|c|c|c|} \hline
        \diaghead{\theadfont aaaaaaaaaa}%
        {$l$}{$n$} & $1$ & $3$ & $5$ \\ 
        \hline
        $2$ &
        $
        0\hookrightarrow\Z_2
        $
        &
        $
        \begin{pmatrix}
            1 & 0 & 0 \\
            0 & 1 & 0 \\
            0 & 0 & 1
        \end{pmatrix}:
        \thead{\Z_4\oplus\Z_2\oplus\Z_3 \\ \hookrightarrow \Z_4\oplus\Z_2\oplus\Z_3}
        $
        & 
        $
        \begin{pmatrix}
            1 & 0 & 0 \\
            1 & 0 & 0 \\
            0 & 1 & 0 \\
            0 & 0 & 1
        \end{pmatrix}
        :\thead{\Z_4\oplus\Z_2\oplus\Z_3 \\ \hookrightarrow \Z_4\oplus\Z_2^2\oplus\Z_3}
        $
        \\ 
        \hline
        \diaghead{\theadfont aaaaaaaaaa}%
        {$l$}{$n$} & $2$ & $4$ & $6$ \\ 
        \hline
        $1$ & 
        $
        \begin{pmatrix}
            1 \\
            1
        \end{pmatrix}
        :\Z_2\hookrightarrow\Z_2^2
        $
        & 
        $
        \begin{pmatrix}
            1 & 0 \\
            1 & 0 \\
            0 & 1
        \end{pmatrix}
        :\Z_2\oplus\Z_3\hookrightarrow\Z_2^2\oplus\Z_3
        $
        & 
        $
        \begin{pmatrix}
            1 & 0 & 0 \\
            0 & 1 & 0 \\
            1 & 0 & 0 \\
            0 & 0 & 1
        \end{pmatrix}
        :\Z_2^3\hookrightarrow\Z_2^4
        $
        \\
        \hline
        $3$ & 
        $
        \begin{pmatrix}
            0 \\
            1
        \end{pmatrix}
        :\Z_3\hookrightarrow\Z_2\oplus\Z_3
        $
        &  
        $
        \begin{pmatrix}
            1 & 0 \\
            1 & 0 \\
            0 & 1
        \end{pmatrix}
        :\Z_2^2\hookrightarrow\Z_2^3
        $
        &  
        $
        \begin{pmatrix}
            1 & 0 & 0 \\
            1 & 0 & 0 \\
            0 & 1 & 0 \\
            0 & 0 & 1
        \end{pmatrix}
        :\Z_2^2\oplus\Z_4\hookrightarrow\Z_2^3\oplus\Z_4
        $
        \\
        \hline
    \end{tabular}
    \caption{Table of matrices $R_{n,l}$ for odd $n+l$}
\label{tab:matrix-R_n,l}
\end{table}

\begin{table}[ht!]
        \centering
        \begin{tabular}{|c|c|}\hline
             Manifold & Congruences \\ \hline
             \thead{$X^{14}\simeq_{\partial}\cp^6\times D^2$ \\ } & \thead{$\overline{\sigma}_{1,2}\equiv0\mod2,$ \\
             $\overline{\sigma}_{3,2}+ 22\overline{\sigma}_{1,2}\equiv0\mod28,$ \\
             $\overline{\sigma}_{4,2}+ \overline{\sigma}_{2,2}\equiv0\mod2,$ \\
             $\overline{\sigma}_{5,2}+ \dfrac{158\overline{\sigma}_{3,2}+1159\overline{\sigma}_{1,2}}{7}\equiv0\mod992$} \\ 
             \hline
             \thead{$X^{16}\simeq_{\partial}\cp^6\times D^4$} & \thead{$\overline{\sigma}_{0,4}\equiv0\mod2$, \\ 
             $\overline{\sigma}_{2,4}-\overline{\sigma}_{0,4}\equiv0\mod28$ \\
             $\overline{\sigma}_{3,4}=0\mod2$, \\ $\overline{\sigma}_{4,4}+ \dfrac{102\overline{\sigma}_{2,4}-109\overline{\sigma}_{0,4}}{7} \equiv0\mod496$ \\
             $\dfrac{351\overline{\sigma}_{4,4}}{31} + \dfrac{6443\overline{\sigma}_{0,4} - 9117\overline{\sigma}_{2,4}}{217}\equiv0\mod 8128$} \\ 
             \hline
             \thead{$X^{18}\simeq_{\partial}\cp^6\times D^6$} & \thead{$\overline{\sigma}_{1,6}\equiv0\mod28,$ \\ 
             $\overline{\sigma}_{2,6}+\overline{\sigma}_{0,6}\equiv0\mod2$, \\
             $\overline{\sigma}_{3,6}+\dfrac{188}{7}\overline{\sigma}_{1,6}\equiv0\mod992$ \\
             $\overline{\sigma}_{4,6}\equiv0\mod2$, \\
             $\overline{\sigma}_{5,6}-\dfrac{23418}{217}\overline{\sigma}_{1,6}+\dfrac{319}{31}\overline{\sigma}_{3,6}\equiv0\mod8128$} \\ \hline
        \end{tabular}
        \caption{Congruences of Splitting Invariants}
        \label{tab:splitting-invariants}
    \end{table}
    
\newpage

\subsection{Spectral Sequence of $[\Sigma^k\cp^n,\coker J_2]$}
\label{subsec:sprectral-sequence-coker-J2}

\begin{sseqdata}[name = cok$J_2$ sseq, xscale=1.3, yscale=0.7, x label={$k$}, y label={$n$}, classes={draw=none}]
               \class["2"](0,4) 
            \class["2"](0,7)
            \class["2^2"](0,8)
            \class["2"](0,9)
            \class["2^2"](0,10)
        \class["2"](1,2)
        \class["2"](1,3)  
        \class["2^2"](1,6)
        \class["2"](1,7)
        \class["8"](1,8) 
        \class["8"](1,9)
        \class["2^2"](1,10)
            \class["2"](2,3)  
            \class["2"](2,6)
            \class["2^2"](2,7)
            \class["2"](2,8)
            \class["2^2"](2,9)
        \class["2"](3,1)
        \class["2"](3,2)  
        \class["2^2"](3,5)
        \class["2"](3,6)
        \class["8"](3,7) 
        \class["8"](3,8)
        \class["2^2"](3,9)
            \class["2"](4,2)  
            \class["2"](4,5)
            \class["2^2"](4,6)
            \class["2"](4,7)
            \class["2^2"](4,8)
        \class["2"](5,0)
        \class["2"](5,1)  
        \class["2^2"](5,4)
        \class["2"](5,5)
        \class["8"](5,6) 
        \class["8"](5,7)
        \class["2^2"](5,8)
            \class["2"](6,1)  
            \class["2"](6,4)
            \class["2^2"](6,5)
            \class["2"](6,6)
            \class["2^2"](6,7)
        \class["2"](7,0)  
        \class["2^2"](7,3)
        \class["2"](7,4)
        \class["8"](7,5) 
        \class["8"](7,6)
        \class["2^2"](7,7)
            \class["2"](8,0)  
            \class["2"](8,3)
            \class["2^2"](8,4)
            \class["2"](8,5)
            \class["2^2"](8,6)
        \class["8\cdot2"](2,10)
        \class["8\cdot2"](4,9)
        \class["8\cdot2"](6,8)
        \class["8\cdot2"](8,7)
            \class["2"](3,10)
            \class["2"](5,9)
            \class["2"](7,8)
            \class["2"](5,10) 
            \class["2"](7,9)
            \class["2"](7,10)
        \d1(1,6)(0,7) \replacesource["2"] \replacetarget[] 
        \d1(3,2)(2,3) \replacesource[] \replacetarget[]
        \d1(3,6)(2,7) \replacesource[] \replacetarget["2"]
        \d1(3,8)(2,9) \replacesource["4"] \replacetarget["2"]
            \d1(4,6)(3,7) \replacesource["2"] \replacetarget["4"]
            \d1(4,8)(3,9) \replacesource["2"] \replacetarget["2"] 
        \d1(5,4)(4,5) \replacesource["2"] \replacetarget[]
        \d1(5,8)(4,9) \replacesource["2"] \replacetarget["4\cdot2"] 
            \d1(6,8)(5,9) \replacesource["8"] \replacetarget[]
        \d1(7,0)(6,1) \replacesource[] \replacetarget[] 
        \d1(7,4)(6,5) \replacesource[] \replacetarget["2"]
        \d1(7,6)(6,7) \replacesource["4"] \replacetarget["2"] 
            \d1(8,4)(7,5) \replacesource["2"] \replacetarget["4"]
            \d1(8,6)(7,7) \replacesource["2"] \replacetarget["2"]
        \d2(1,2)(0,4) \replacesource[] \replacetarget[]
        \d2(3,5)(2,7) \replacesource["2"] \replacetarget[]
        \d2(3,8)(2,10)  \replacetarget["4\cdot2"]
        \d2(4,6)(3,8) \replacesource[] \replacetarget[]
        \d2(5,6)(4,8) \replacesource["4"] \replacetarget[]
        \d2(6,5)(5,7) \replacesource[] \replacetarget["4"]
        \d2(6,6)(5,8) \replacesource[] \replacetarget[]
        \d2(7,5)(6,7) \replacesource["2"] \replacetarget[]
        \d2(7,6)(6,8) \replacesource[] \replacetarget["2"]
        \d2(8,5)(7,7) \replacesource[] \replacetarget[]
        \d3(2,6)(1,9) \replacesource[] \replacetarget["4"]
        \d3(5,4)(4,7) \replacesource[] \replacetarget[]
        \d4(1,6)(0,10) \replacesource[] \replacetarget["2"]
        \d4(3,5)(2,9) \replacesource[] \replacetarget[]
\end{sseqdata}
\begin{center}
    \printpage[name = cok$J_2$ sseq, title = 1st page, page=1] \\
    \printpage[name = cok$J_2$ sseq, title = 2nd page, page=2] \\
    \printpage[name = cok$J_2$ sseq, title = 3rd page, page=3] \\
    \printpage[name = cok$J_2$ sseq, title = 4th page, page=4] 
    
\end{center}

\newpage

\subsection{Spectral Sequence of $[\Sigma^k\cp^n,\coker J_3]$}
\label{subsec:sprectral-sequence-coker-J3}

\begin{sseqdata}[name = 3-primary sseq, xscale=1.3, yscale=0.7, x label={$k$}, y label={$n$}, classes={draw=none}]
            \class["\Z_3"](0,6) 
        \class["\Z_3"](1,4) \class["\Z_3"](1,9)
         \class["\Z_3"](2,5) 
        \class["\Z_3"](3,3) \class["\Z_3"](3,8)
        \class["\Z_3"](4,4) 
        \class["\Z_3"](5,2) \class["\Z_3"](5,7)
        \class["\Z_3"](6,3) 
        \class["\Z_3"](7,1) \class["\Z_3"](7,6)
         \class["\Z_3"](8,2)
        \class["\Z_3"](2,10)
        \class["\Z_3"](4,9)
        \class["\Z_3"](6,8)
        \class["\Z_3"](8,7)
         \class["\Z_3"](5,10) \class["\Z_3"](7,9)
        \d2(1,4)(0,6) \replacesource[] \replacetarget[]
        \d2(3,3)(2,5) \replacesource[] \replacetarget[] 
        \d2(7,1)(6,3) \replacesource[] \replacetarget[]
        \d2(5,7)(4,9) \replacesource[] \replacetarget[]
        \d2(7,6)(6,8) \replacesource[] \replacetarget[]
        \d4(4,4)(3,8) \replacesource[] \replacetarget[]
\end{sseqdata}
\begin{center}
    \printpage[name = 3-primary sseq, title = 2nd page,  page=2] \\
    \printpage[name = 3-primary sseq, title = 4th page,  page=4]
\end{center}

\begin{table}[ht!]
    \centering
    \begin{tabular}{|c|c|}
        \hline
         $r$ & $U(m,r)$  \\ \hline
         $2$ & $2/(m,2)$ \\ \hline
         $3$ & \thead{\normalsize $16/(m,8)(m,2)(m+5,8)$ \\ \normalsize $3/(m,3)$} \\ \hline
         $4$ & \thead{\normalsize $16/(m,8)(m,2)(m+5,8)$ $(*)$ \\ \normalsize $3/(m,3)$} \\ \hline
         $5$ & \thead{\normalsize $128/(m,64)(m+2,16)(m+1,16)(m+3,8)$ \\ \normalsize $9/(m,9)(m+1,3)$ \\ \normalsize $5/(m,5)$} \\ \hline
          & $(*)$ Multiplied by $2$ if $m=6$ \\ \hline
    \end{tabular}
    \caption{Complex James Numbers $U(m,r)$}
    \label{tab:james-numbers}
\end{table}

\newpage

\begin{table}[ht!]
    \centering
    \begin{tabular}{|c|c|} \hline
        $n$ & $V^2(\cp^n\times D^k)$  \\ \hline
        $1$ & $1\times1$ \\ \hline
        $2$ & $(1+x^2)\times1$ \\ \hline
        $3$ & $1\times1$ \\ \hline
        $4$ & $(1+x^2+x^4)\times1$ \\ \hline
        $5$ & $(1+x^4)\times1$ \\ \hline
        $6$ & $(1+x^2+x^6)\times1$ \\ \hline
    \end{tabular}
    \caption{Table of squares of Wu classes of $\cp^n\times D^k$}
    \label{tab:Wu-classes}
\end{table}

\begin{table}[ht!]
    \centering
    \begin{tabular}{|c|c|c|c|} \hline
        \diaghead{\theadfont aaaaaaaa}%
        {$l$}{$n$} & $1$ & $3$ & $5$ \\ \hline
        $2$ & 
        $\LARGE\blank$
        &
        $
        \begin{pmatrix}
            80 & 0 \\
            0 & 0
        \end{pmatrix}
        $
        & 
        $
        \begin{pmatrix}
            80 & 0 \\
            54 & 728 
        \end{pmatrix}
        $
        \\ \hline
        \diaghead{\theadfont aaaaaaaa}%
        {$l$}{$n$} & $2$ & $4$ & $6$ \\ \hline
        $1$ &
        $
        \begin{pmatrix}
            8
        \end{pmatrix}
        $
        &
        $
        \begin{pmatrix}
            8 & 0 \\
            12 & 80
        \end{pmatrix}
        $
        &
        $
        \begin{pmatrix}
            8 & 0 & 0 \\
            12 & 80 & 0 \\
            3 & 162 & 728
        \end{pmatrix}
        $
        \\ \hline
        $3$ &
        $
        \begin{pmatrix}
            80
        \end{pmatrix}
        $
        &
        $
        \begin{pmatrix}
            80 & 0 \\
            108 & 728
        \end{pmatrix}
        $
        & 
        $
        \begin{pmatrix}
            80 & 0 & 0 \\
            108 & 728 & 0 \\
            27 & 1458 & 6560
        \end{pmatrix}
        $
        \\ \hline
    \end{tabular}
    \caption{Table of operations $[\Sigma^{2l}\cp^n,\Psi^3-\id]$}
    \label{tab:Adams-id-operations}
\end{table}

\begin{table}[ht!]
    \centering
    \begin{tabular}{|c|c|c|c|} \hline
        \diaghead{\theadfont aaaaaaaa}%
        {$l$}{$n$} & $1$ & $3$ & $5$ \\ \hline
        $2$ & 
        $\LARGE\blank$
        &
        \thead{$\xi_1\mapsto3\mu_2+\{0,\mu_2\mu_0\}$ \\ $\in[\Sigma^4\cp^3,BO]_{(2)}/\Z_2\{\mu_2\mu_0\}$}
        & 
        \thead{
            $\xi_1\mapsto3\mu_2+\frac{5}{13}\mu_2\mu_0$ \\
            $\xi_2\mapsto\frac{9}{13}\mu_2\mu_0$
        }
        \\ \hline
        \diaghead{\theadfont aaaaaaaa}%
        {$l$}{$n$} & $2$ & $4$ & $6$ \\ \hline
        $1$ &
        $\xi_1\mapsto3\mu_1$
        &
        \thead{
            $\xi_1\mapsto3\mu_1+2\mu_1\mu_0$ \\
            $\xi_2\mapsto3\mu_1\mu_0$
        }   
        &
        \thead{
            $\xi_1\mapsto3\mu_1+2\mu_1\mu_0$ \\
            $\xi_2\mapsto3\mu_1\mu_0-\frac{3}{13}\mu_1\mu_0^2$ \\
            $\xi_3\mapsto\frac{9}{13}\mu_1\mu_0^2$
        }   
        \\ \hline
        $3$ &
        $\xi_1\mapsto3\mu_3$
        &
        \thead{
            $\xi_1\mapsto3\mu_3+\frac{1}{13}\mu_3\mu_0$ \\
            $\xi_2\mapsto\frac{9}{13}\mu_3\mu_0$
        }   
        & 
        \thead{
            $\xi_1\mapsto3\mu_3+\frac{1}{13}\mu_3\mu_0+\frac{6}{533}\mu_3\mu_0^2$ \\
            $\xi_2\mapsto\frac{9}{13}\mu_3\mu_0-\frac{63}{533}\mu_3\mu_0^2$ \\
            $\xi_3\mapsto\frac{3}{41}\mu_3\mu_0^2$
        }   
        \\ \hline
    \end{tabular}
    \caption{\centering Table of images of generators of $\ker[\Sigma^{2l}\cp^n,J]_{(2)}$ under $[\Sigma^{2l}\cp^n,\Psi^3-\id]_{(2)}^{-1}$}
    \label{tab:inverse-Adams-id-generators}
\end{table}

\begin{table}[ht!]
    \centering
    \begin{tabular}{|c|c|}
        \hline
        $\Sigma^{2}\cp^6$ & \thead{$W(\mu_1)=1+(x+x^2+x^3+x^4+x^5+x^6)\times y$ \\ $W(\mu_1\mu_0)=W(\mu_1\mu_0^2)=1$} \\ \hline
        $\Sigma^{4}\cp^6$ & \thead{$W(\mu_2)=1+(x^2+x^4+x^6)\times y^2$ \\ $W(\mu_2\mu_0)=W(\mu_2\mu_0^2)=1$} \\ \hline
        $\Sigma^{6}\cp^6$ & $W(\mu_3)=W(\mu_3\mu_0)=W(\mu_3\mu_0^2)=1$ \\ \hline
    \end{tabular}
    \caption{\centering Table of total Stiefel-Whitney classes of generators of $[\Sigma^{2l}\cp^6,BO]$}
    \label{tab:S-W-classes-generators}
\end{table}

\newpage

%% file: bibliography.bib
@article{Fujii,
  title={{KO}-groups of projective spaces},
  author={Fujii, Michikazu},
  journal={Osaka J. Math},
  volume={4},
  pages={141--149},
  year={1967}
}

@article{adamsIV,
title = {On the groups J(X)—IV},
journal = {Topology},
volume = {5},
number = {1},
pages = {21-71},
year = {1966},
author = {J.F. Adams}
}

@article{Adams_vect,
title = {Vector fields on spheres},
journal = {Topology},
volume = {1},
number = {1},
pages = {63-65},
year = {1962},
author = {J. F. Adams}
}

@book{mccleary2001user,
  title={A user's guide to spectral sequences},
  author={McCleary, John},
  number={58},
  year={2001},
  publisher={Cambridge University Press}
}

@article{Imanishi,
  title={{Unstable Homotopy Groups of Classical Groups} (odd primary components).},
  author={Hideki Imanishi},
  journal={Journal of Mathematics of Kyoto University},
  year={1967},
  volume={7},
  pages={221-243},
}

@article{toda1959p,
  title={$ p $-primary components of homotopy groups IV. Compositions and toric constructions},
  author={Toda, Hirosi},
  journal={Memoirs of the College of Science, University of Kyoto. Series A: Mathematics},
  volume={32},
  number={2},
  pages={297--332},
  year={1959},
  publisher={Duke University Press}
}

@book{toda1962composition,
  title={Composition methods in homotopy groups of spheres},
  author={Toda, Hiroshi},
  number={49},
  year={1962},
  publisher={Princeton University Press}
}

@article{Rourke,
 author = {C. P. Rourke and D. P. Sullivan},
 journal = {Annals of Mathematics},
 number = {3},
 pages = {397--413},
 publisher = {[Annals of Mathematics, Trustees of Princeton University on Behalf of the Annals of Mathematics, Mathematics Department, Princeton University]},
 title = {On the Kervaire Obstruction},
 volume = {94},
 year = {1971}
}

@article{Hambleton,
author = {Hambleton, I. and Milgram, Richard and Taylor, Laurence and Williams, Bruce},
year = {1988},
title = {Surgery with Finite Fundamental Group},
volume = {56},
journal = {Proceedings of the London Mathematical Society},
}

@article{brumfiel_transfer,
	title = {Evaluation of the transfer and the universal surgery classes},
	volume = {32},
	number = {2},
	journal = {Inventiones mathematicae},
	author = {Brumfiel, G. and Madsen, I.},
	year = {1976},
	pages = {133--169},
}

@article{brumfiel_PL,
 author = {G. Brumfiel and I. Madsen and R. J. Milgram},
 journal = {Annals of Mathematics},
 number = {1},
 pages = {82--159},
 publisher = {[Annals of Mathematics, Trustees of Princeton University on Behalf of the Annals of Mathematics, Mathematics Department, Princeton University]},
 title = {PL Characteristic Classes and Cobordism},
 volume = {97},
 year = {1973}
}

@Book{Madsen,
  author    = {Madsen, Ib and Milgram, R. James},
  publisher = {Princeton University Press, Princeton, NJ},
  title     = {The classifying spaces for surgery and cobordism of manifolds},
  year      = {1979},
  series    = {Ann. {Math}. {Stud}.},
  volume    = {92},
}

@article{imaoka1984stable,
  title={On the Stable Hurewicz Image of Some Stunted Projective Spaces, I},
  author={Imaoka, Mitsunori and Morisugi, Kaoru},
  journal={Publications of the Research Institute for Mathematical Sciences},
  volume={20},
  number={4},
  pages={839--852},
  year={1984},
  publisher={Research Institute forMathematical Sciences}
}

@inproceedings{lundell,
  title={Concise tables of James numbers and some homotopy of classical Lie groups and associated homogeneous spaces},
  author={Lundell, Albert T},
  booktitle={Algebraic Topology Homotopy and Group Cohomology: Proceedings of the 1990 Barcelona Conference on Algebraic Topology, 1990},
  pages={250--272},
  year={2006},
  organization={Springer}
}

@article{Atiyah_Hirzebruch, 
    title={Bott Periodicity and the Parallelizability of the spheres}, 
    volume={57}, 
    number={2}, 
    journal={Mathematical Proceedings of the Cambridge Philosophical Society}, 
    author={Atiyah, M. F. and Hirzebruch, F.}, 
    year={1961}, 
    pages={223–226}
}

@book{milnor1974characteristic,
  title={Characteristic classes},
  author={Milnor, John Willard and Stasheff, James D},
  number={76},
  year={1974},
  publisher={Princeton university press}
}

@book{hatcher2002algebraic,
  title={Algebraic Topology},
  author={Hatcher, A.},
  isbn={9780521795401},
  lccn={00065166},
  series={Algebraic Topology},
  url={https://books.google.sk/books?id=BjKs86kosqgC},
  year={2002},
  publisher={Cambridge University Press}
}

@article{article,
author = {Mahowald, Mark and Ravenel, Douglas},
year = {1984},
pages = {},
title = {Toward a Global Understanding of the Homotopy Groups of Spheres},
isbn = {9780821850633}
}

@article{quillen1971adams,
  title={The Adams conjecture},
  author={Quillen, Daniel},
  journal={Topology},
  volume={10},
  number={1},
  pages={67--80},
  year={1971}
}

@Article{Brumfiel1971,
  author   = {Brumfiel, G.},
  journal  = {Commentarii Mathematici Helvetici},
  title    = {Homotopy equivalences of almost smooth manifolds},
  year     = {1971},
  pages    = {381--407},
  volume   = {46},
}

@misc{kasilingam2026diffeomorphismclassificationsmoothstructures,
      title={Diffeomorphism Classification of Smooth Structures and Tangential Homotopy Types of $\mathbb{C}P^m$ for $5 \le m \le 8$}, 
      author={Ramesh Kasilingam},
      year={2026},
      eprint={2604.27521},
      archivePrefix={arXiv},
      primaryClass={math.AT},
      url={https://arxiv.org/abs/2604.27521}, 
}

@misc{kaluzny2026highersmoothsurgerystructure,
      title={Higher Smooth Surgery Structure Sets of Complex Projective Spaces, Part I}, 
      author={Samuel Kalužný and Tibor Macko},
      year={2026},
      eprint={2605.04817},
      archivePrefix={arXiv},
      primaryClass={math.AT},
      url={https://arxiv.org/abs/2605.04817}, 
}

@article{Kachi2001SomeCG,
  title={Some Cohomotopy Groups of Suspended Projective Planes},
  author={Hideyuki Kachi and Jun Mukai and T. Nozaki and Yukimasa Sumita and Dai Tamaki},
  journal={Mathematical journal of Okayama University},
  year={2001},
  volume={43}
}

@article{mukai1969stable,
  title={On the stable homotopy of a Z$\backslash$sb2-Moore space},
  author={Mukai, Juno},
  year={1969}
}

@book{kochman2006stable,
  title={Stable homotopy groups of spheres: a computer-assisted approach},
  author={Kochman, Stanley O},
  year={2006},
  publisher={Springer}
}

@article{Behrens-Hill-Hopkins-Mahowald,
author = {Behrens, M. and Hill, M. and Hopkins, M. J. and Mahowald, M.},
title = {Detecting exotic spheres in low dimensions using coker J},
journal = {Journal of the London Mathematical Society},
volume = {101},
number = {3},
pages = {1173-1218},
doi = {https://doi.org/10.1112/jlms.12301},
year = {2020}
}

@incollection{weiss-williams-automorph,
 author = {Weiss, Michael and Williams, Bruce},
 title = {Automorphisms of manifolds},
 booktitle = {Surveys on surgery theory. Vol. 2: Papers dedicated to C. T. C. Wall on the occasion of his 60th birthday},
 isbn = {0-691-08815-2; 0-691-08814-4},
 pages = {165--220},
 year = {2001},
 publisher = {Princeton, NJ: Princeton University Press},
 }

@article{sasao-homotopy-of-map-cp,
author = {Sasao, Seiya},
title = {The Homotopy of Map(Cpm, CPn)},
journal = {Journal of the London Mathematical Society},
volume = {s2-8},
number = {2},
pages = {193-197},
doi = {https://doi.org/10.1112/jlms/s2-8.2.193},
year = {1974}
}

@book{brown-group-cohomology,
 author = {Brown, Kenneth S.},
 title = {Cohomology of groups},
 fseries = {Graduate Texts in Mathematics},
 series = {Grad. Texts Math.},
 issn = {0072-5285},
 volume = {87},
 year = {1982},
 publisher = {Springer, Cham},
 language = {English}
}

@article{cerf1970,
  title={La stratification naturelle des espaces de fonctions diff{\'e}rentiables r{\'e}elles et le th{\'e}oreme de la pseudo-isotopie},
  author={Cerf, Jean},
  journal={Publications Math{\'e}matiques de l'IH{\'E}S},
  volume={39},
  pages={5--173},
  year={1970}
}

@book{burghelea2006,
  title={Groups of automorphisms of manifolds},
  author={Burghelea, Dan and Lashof, Richard and Rothenberg, Melvin},
  year={2006},
  publisher={Springer}
}

@misc{kreck2024,
      title={Mapping class group of manifolds which look like $3$-dimensional complete intersections}, 
      author={Matthias Kreck and Yang Su},
      year={2024},
      eprint={2009.08054},
      archivePrefix={arXiv},
      primaryClass={math.GT},
      url={https://arxiv.org/abs/2009.08054}, 
}

@misc{randal-williams,
  author       = {O. Randal-Williams},
  title        = {Answer to the MathOverflow question: Is the mapping class group of $\cp^n$ known?},
    note         = {Accessed: 2026-08-26},
  howpublished = {\url{https://mathoverflow.net/questions/349319/is-the-mapping-class-group-of-bbbcpn-known}}
}
